\documentclass{amsart}

\DeclareFontFamily{OT1}{pzc}{}
\DeclareFontShape{OT1}{pzc}{m}{it}{<-> s * [1.2] pzcmi7t}{}
\DeclareMathAlphabet{\mathpzc}{OT1}{pzc}{m}{it}

\usepackage{bm}
\usepackage{amsmath}
\usepackage{amssymb}
\usepackage{amsthm}
\usepackage{mathtools}
\usepackage{stmaryrd}
\usepackage{tikz}
\usepackage{subcaption}
\usepackage{pgfplots}
\usepackage{pgfplotstable}

\numberwithin{equation}{section}

\newcommand{\Div}[0]{\mathbf{div}\,}
\newcommand{\bs}[1]{\boldsymbol{#1}}
\newcommand{\Sig}[0]{\bs{\sigma}}
\newcommand{\Tau}[0]{\bs{\tau}}
\newcommand{\Eps}[0]{\bs{\varepsilon}}
\newcommand{\U}[0]{\bs{u}}
\newcommand{\V}[0]{\bs{v}}
\newcommand{\W}[0]{\bs{w}}
\newcommand{\bZ}[0]{\bs{z}}
\newcommand{\F}[0]{\bs{f}}
\newcommand{\Z}[0]{\bs{0}}
\newcommand{\N}[0]{\bs{n}}
\newcommand{\bQ}[0]{\bs{q}}
\newcommand{\T}[0]{\bs{t}}
\newcommand{\X}[0]{\bs{x}}
\newcommand{\dx}[0]{\,\mathrm{d}x}
\newcommand{\dy}[0]{\,\mathrm{d}y}

\newcommand{\jump}[1]{\left\llbracket #1 \right\rrbracket}
\newcommand{\mean}[1]{\left\{ \hspace{-0.12cm} \left\{ #1 \right\} \hspace{-0.12cm} \right\}}
\newcommand{\dual}[1]{\left\langle #1 \right\rangle}
\newcommand{\Bf}[0]{\mathcal{B}}
\newcommand{\Lf}[0]{\mathcal{L}}
\newcommand{\Sf}[0]{\mathcal{S}_h}
\newcommand{\vertiii}[1]{{\left\vert\kern-0.25ex\left\vert\kern-0.25ex\left\vert #1 
    \right\vert\kern-0.25ex\right\vert\kern-0.25ex\right\vert}}
    \newcommand{\enorm}[1]{\vertiii{#1}}

\newcommand{\Ch}[0]{\mathcal{C}_h}
\newcommand{\Eh}[0]{\mathcal{E}_h}
\newcommand{\Gh}[0]{\mathcal{G}_h}
\newcommand{\Nh}[0]{\mathcal{N}_h}

\newcommand{\Gcap}[0]{\mathcal{G}_h^{12}}

\newcommand{\sinu}{\sigma_{n}(\U_i)}
\newcommand{\sihnu}{\sigma_{i,n}(\U_{i,h})}
\newcommand{\sihnv}{\sigma_{i,n}(\V_{i,h})}
\newcommand{\sihnw}{\sigma_{i,n}(\W_{i,h})}

\newcommand{\un}{u_n}
\newcommand{\wn}{w_n}
\newcommand{\uhn}{u_{h,n}}
\newcommand{\vhn}{v_{h,n}}
\newcommand{\vn}{v_n}

\newcommand{\sonehnu}{\sigma_{1,n}(\U_{1,h})}
\newcommand{\stwohnu}{\sigma_{2,n}(\U_{2,h})}
\newcommand{\sonehnv}{\sigma_{1,n}(\V_{1,h})}

\newcommand{\Q}{H^{\frac12}(\Gamma)}
\newcommand{\mQ}{H^{-\frac12}(\Gamma)}

 \newcommand{\situ}{\Sig_{i,t}(\U_i)} 
  \newtheorem{thm}{Theorem}
 
  \newcommand{\h}[0]{\mathpzc{h}}

 \newcommand{\bV}{\bm V}
\newtheorem{lem}{Lemma}

\newtheorem{prob}{Problem}
\newtheorem{rem}{Remark}
\newtheorem{nitsche}{Nitsche formulation}

\numberwithin{equation}{section} 
\numberwithin{rem}{section} 
\numberwithin{lem}{section} 
\numberwithin{thm}{section}

\title{On Nitsche's method for   elastic contact problems}
\author{Tom Gustafsson}
\address{Department of Mathematics and Systems Analysis, Aalto University, 00076 Aalto, Finland}
\email{tom.gustafsson@alumni.aalto.fi.}
\author{Rolf Stenberg}
\address{Department of Mathematics and Systems Analysis, Aalto University, 00076 Aalto, Finland}
\email{rolf.stenberg@aalto.fi}
\author{Juha Videman}
\address{CAMGSD/Departamento de Matem\'atica, Universidade de Lisboa, Universidade de Lisboa, 1049-001 Lisbon, Portugal}
\email{jvideman@math.tecnico.ulisboa.pt}
\thanks{The financial support from the Portuguese government through FCT (Funda\c{c}\~ao para a Ci\^encia e a Tecnologia), I.P., under the projects PTDC/MAT-PUR/28686/2017 and UTAP-EXPL/MAT/0017/2017, is gratefully acknowledged.}

\begin{document}

\maketitle

\begin{abstract}
 We show quasi-optimality and a posteriori error estimates for the frictionless contact problem between two elastic bodies with a zero-gap function. The analysis is based on interpreting Nitsche's method as a stabilised finite element method for which the error estimates can be obtained with minimal regularity assumptions and without the saturation assumption. We present three different Nitsche's mortaring techniques for the contact boundary each corresponding to a different stabilising term. Our numerical experiments show the robustness of Nitsche's method and corroborates the efficiency of the a posteriori error estimators.
\end{abstract}
 
\section{introduction}

In this paper, we analyse the Nitsche method for elastic contact problems. Over the last decade, this method has been studied by a number of authors, see, e.g., \cite{CH13,french-survey,CFHPR18,ChoulyRabii}, and shown to be a robust and  efficient method. The advantages are an easy implementation based on the displacement variables only and, when compared to mixed methods with Lagrange multipliers, the absence of  an "inf-sup" stability condition which renders a symmetric positive definite system instead of one with  a saddle point structure.

From a theoretical point of view, the previously mentioned works suffer from two shortcomings. First, for the problem posed in $H^1$, the solution is typically assumed to be in $H^s$, with $s>3/2$. Second, the a posteriori error analyses are often based on a non-rigorous saturation assumption.

We have addressed these issues in our recent articles, cf. \cite{GSV17,GSV17a}. Our approach dates back to \cite{Stenberg1995} where different ways to enforce weakly the Dirichlet boundary conditions were discussed in the context of the so called stabilised mixed methods \cite{BaHu1,BaHu2}  wherein the bilinear form of the original mixed finite element method is augmented with a properly weighted residual term to ensure stability. In \cite{Stenberg1995}, it was shown that the local elimination of the Lagrange multiplier leads essentially to a method introduced by Nitsche in the early age of the finite element analysis \cite{Nitsche}. Since Nitsche's method is straightforward both to analyse (under the additional smoothness assumption) and to implement, we started to advocate it, in particular for contact problems, cf. \cite{Stenberg1998,BHS}.

What we have realised recently is that one should take full advantage of the relation between Nitsche's and stabilised method when analysing the former. In fact, we were able to get rid of both the smoothness and the saturation assumption for the membrane obstacle problem in \cite{GSV17}. In this paper, we will continue on this path and perform an error analysis, both quasi-optimality and a posteriori, for a simplified two-body contact problem without friction.  Besides the theoretical improvements, we present three versions of the Nitsche's method where the changes in the material parameters between the bodies are taken into account. The simplest is a typical "master-slave" approach where the contact surface of the stiffer body is chosen as the master part and the slave surface is then mortared by the Nitsche's technique. In the two other variants, the material parameters appear as weights in the Nitsche formulation so that the methods decide by themselves which part is the master and which is the slave. In order to simplify the notation, analysis and implementation of the adaptive methods, we assume that the elastic bodies are initially in full contact, see, {\sl e.g.}, \cite{HW05}, and leave the case with a non-vanishing initial gap between the elastic bodies for a future work.

Although our analysis is built upon  our  earlier works, cf. \cite{GSV17,GSVmortar}, we will present proofs of all the main theorems. 
We also note that the elastic contact problem literature is vast and therefore we only refer to the review paper \cite{W11}, and to all the references therein, for the analysis and application of finite element methods arising from mixed formulations and to \cite{KVW15,chouly2017residual}, and to all the references therein, for the a posteriori error analyses of contact problems.  We end the paper by presenting results of our computational experiments.

\section{The contact problem}

Let $\Omega_i \subset \mathbb{R}^d$, $i=1,2$, $d\in\{2,3\}$, denote two elastic bodies in their reference configuration and assume that the bodies are initially in contact.
Moreover,  assume that $\Omega_i $ are polygonal (polyhedral) domains and denote by  \mbox{$\Gamma = \partial \Omega_1 \cap \partial
\Omega_2$} their common boundary.
The boundary $\partial \Omega_i $  is split into three disjoint sets
$ \Gamma_{D,i}, \Gamma_{N,i}$ and $\Gamma_{C,i}$, with $ \Gamma_{D,i} $ denoting the part where homogeneous Dirichlet data is given, $ \Gamma_{N,i} $ the part  with a Neumann boundary condition and $\Gamma_{C,i} $ the part where contact can occur, see Figure~\ref{fig:twobodynot}.

Letting $\U_i : \Omega_i \rightarrow
\mathbb{R}^d$, $i = 1,2$, be the displacement of the body $\Omega_i$, the
infinitesimal strain tensor is defined as
\begin{equation}
    \Eps(\U_i) = \frac12\Big(\nabla \U_i + (\nabla \U_i)^T\Big).
\end{equation}
We assume  homogenous isotropic bodies and a plain strain problem in the two dimensional case. The stress tensor is thus given by
\begin{equation}
    \Sig_i(\U_i) = 2 \mu_i\,\Eps(\U_i) + \lambda_i\,\mathrm{tr}\,\Eps(\U_i) \boldsymbol{I},
\end{equation}
where $\mu_i>0$ is the shear modulus and $\lambda_i$ the second Lam\'e parameter of the body
$\Omega_i$ and $\boldsymbol{I}$ denotes the $d$-dimensional identity tensor.   We will exclude the possibility that the materials are nearly incompressible and hence it holds 
$\lambda_i \lesssim \mu_i$. (For nearly incompressible materials the standard approach of reformulating the problem in mixed form \cite{MR1115205} should be used.)

By $\N_i \in \mathbb{R}^d$ we denote the outward unit normal to
$\partial \Omega_i$, and define $\N =  \N_1 =- \N_2$.  In what follows, $\T$ denotes any unit vector that satisfies $\N \cdot \T = 0$.

We decompose the traction vector on $\partial \Omega_i$,   $\Sig_i(\U_i)\N_i $, into its normal and tangential parts, viz.
\begin{equation}
  \Sig_i(\U_i)\N_i =\Sig_{i,n}(\U_i)+\Sig_{i,t}(\U_i).
  \end{equation}
For the scalar normal  tractions we  use the sign convention
\begin{equation}
\sigma_{1,n}(\U_1) = \Sig_{1,n}(\U_1)\cdot \N_1 ,
\end{equation}
and
\begin{equation}
\sigma_{2,n}(\U_2) = -\Sig_{2,n}(\U_2)\cdot \N_2,
\end{equation}
and note that on $\Gamma$ these tractions are either both zero or continuous and compressive, i.e. it holds that
\begin{equation}
\sigma_{1,n}(\U_1)=\sigma_{2,n}(\U_2), \quad \sigma_{i,n}(\U_i) \leq 0, \ i=1,2. 
\end{equation}

The physical non-penetration constraint on $\Gamma$  reads as 
\begin{equation}
      \U_1 \cdot \N_1 + \U_2 \cdot \N_2 \leq 0,
\end{equation}
which, defining 
\begin{equation}
\un= -( \U_1 \cdot \N_1 + \U_2 \cdot \N_2)
\end{equation}
can be written as 
\begin{equation}
\jump{\un}\geq 0,
\end{equation}
where $\jump{\cdot }$ denotes the jump over $\Gamma$.

We thus have the following problem.

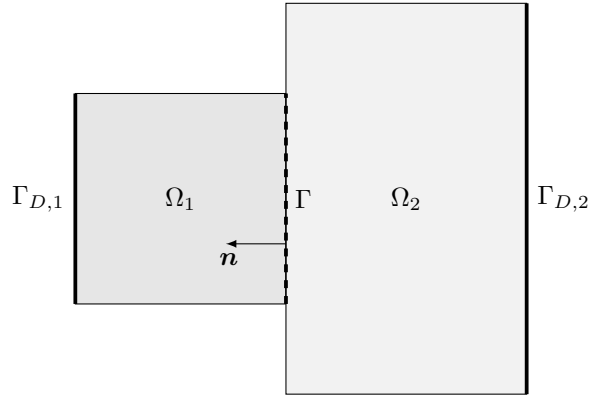
\begin{figure}[h!]
    \centering
    \begin{tikzpicture}[scale=0.8]
        \draw[fill=gray!20] (0.5, 0.5) rectangle (4, 4) node[pos=.5] {$\Omega_1$};
        \draw[fill=gray!10] (4, -1) rectangle (8, 5.5) node[pos=.5] {$\Omega_2$};
        \draw (4, 0.5) -- (4, 4) node[midway, anchor=west] {$\Gamma$};
        \draw[line width=0.5mm] (0.5, 0.5) -- (0.5, 4) node[midway, anchor=east] {$\Gamma_{D,1}$};
        \draw[line width=0.5mm] (8, -1) -- (8, 5.5) node[midway, anchor=west] {$\Gamma_{D,2}$};
        \draw[line width=0.5mm, dashed] (4, 0.5) -- (4, 4);
        \draw[-latex] (4, 1.5) -- (3, 1.5) node[pos=0.95, anchor=north] {$\boldsymbol{n}$};
    \end{tikzpicture}
    \caption{Notation for the elastic contact problem.}
    \label{fig:twobodynot}
\end{figure}

\begin{prob}[Strong formulation]
    Find $\U_i : \Omega_i \rightarrow \mathbb{R}^d$, $i = 1,2$, $d\in\{2,3\}$, such that
 \begin{equation}  \label{strong}
  \begin{aligned} 
        -\Div \Sig_i(\U_i)&=\F_i \quad && \text{in $\Omega_i$,} \\
        \U_i&=\Z \quad && \text{on $\Gamma_{D,i}$,} \\
        \Sig_i(\U_i) \N_i &=\Z \quad && \text{on $\Gamma_{N,i}$,} \\
        \situ &= \Z && \text{on $\Gamma$,} \\
       \sigma_{1,n}(\U_1) - \sigma_{2,n}(\U_2)&=0 && \text{on $\Gamma$,} \\
        \jump{\un}   &\geq 0 && \text{on $\Gamma$,} \\
    \sigma_{i,n}(\U_i)&\leq 0 && \text{on $\Gamma$,} \\
          \jump{\un}  \sigma_{i,n}(\U_i) &= 0 && \text{on $\Gamma$,}
    \end{aligned}
    \end{equation}
    \label{prob:strong}
    where $\F_i\in  [L^2(\Omega_i)]^d$ denotes the volume force on $\Omega_i$. 
\end{prob}
Letting $\lambda =- {\sigma_{1,n}(\U_1)}=-{\sigma_{2,n}(\U_2)}$ denote a Lagrange multiplier associated with the contact constraint,
we obtain an equivalent mixed formulation in which the normal traction on the contact surface is an independent unknown.
\begin{prob}[Mixed formulation]
    Find $\U_i : \Omega_i \rightarrow \mathbb{R}^d$, $i = 1,2$, $d\in\{2,3\}$, and $\lambda : \Gamma \rightarrow \mathbb{R}$, such that
 \begin{equation}  \label{mixedstrong}
  \begin{aligned} 
        -\Div \Sig_i(\U_i)&=\F_i \quad && \text{in $\Omega_i$,} \\
        \U_i&=\Z \quad && \text{on $\Gamma_{D,i}$,} \\
        \Sig_i(\U_i) \N_i &=\Z \quad && \text{on $\Gamma_{N,i}$,} \\
        \situ &= \Z && \text{on $\Gamma$,}\\
      \lambda + {\sigma_{1,n}(\U_1)}    &=0,  &&\text{on $\Gamma$,} 
            \\
      \lambda +  {\sigma_{2,n}(\U_2)    }  &=0,  &&\text{on $\Gamma$,} 
       \\
        \jump{\un}   &\geq 0 && \text{on $\Gamma$,} \\
       \lambda&\geq  0 && \text{on $\Gamma$,} \\
       \jump{\un} \lambda &= 0 && \text{on $\Gamma$.}
    \end{aligned}
    \end{equation}
    \label{prob:strong}
     
\end{prob}

To present a variational formulation for Problem~\ref{prob:strong},
we introduce function spaces for the displacements 
\begin{equation}
   \bV_i=\{\W_i   \in[H^1(\Omega_i)]^d : \W_i |_{\Gamma_{D,i}} = \Z\}   , 
\end{equation}
and equip them with the usual norms $\|\cdot\|_{1,\Omega_i}$. Moreover, we write $\bV = \bV_1 \times \bV_2$ and assume that
$\Gamma$ is a compact subset of $\partial \Omega_i \setminus \Gamma_{D,i}$ for  $i=1,2$.  
Thus the normal components of the displacement traces on the contact zone are in
  $H^{\frac12}(\Gamma)$  with the intrinsic norm in   $H^{\frac12}(\Gamma)$  defined by (cf., e.g., \cite{tartar})
\begin{equation}
    \|w\|_{\frac12,\Gamma}^2 = \|w\|_{0,\Gamma}^2 + \int_\Gamma \int_\Gamma \frac{|w(x) - w(y)|^2}{|x-y|^d} \dx \dy.
\end{equation}

The inequality constraint on $\Gamma$ is imposed by the  Lagrange multiplier which
belongs to $\mQ$, the topological dual of $\Q$, i.e.~$\mQ = \Q^\prime$. The duality pairing   is denoted by 
$\dual{\cdot,\cdot} : \Q \times\mQ \rightarrow \mathbb{R}$, and the norm is then
 \begin{equation}
    \|\xi\|_{-\frac12, \Gamma}  = \sup_{w \in W} \frac{\langle w, \xi \rangle}{\|w\|_{\frac12,\Gamma}}.
    \end{equation}
Moreover, we define the positive part of $\mQ$ as
\begin{equation}\label{pospart}
    \varLambda=\{\xi \in \mQ : \dual{ w, \xi} \geq 0 ~~ \forall w \in \Q,~w \geq 0 ~\text{a.e.~on $\Gamma$}\}
\end{equation}
and introduce the  bilinear and linear forms  
\begin{equation}
    \Bf(\W,\xi; \V,\eta)  = \sum_{i=1}^2 (\Sig_i(\W_i),\Eps(\V_i))_{\Omega_i} - \dual{\jump{\vn}  , \xi} - \dual{\jump{\wn} , \eta},
    \end{equation}
    and
    \begin{equation}
    \Lf(\V) = \sum_{i=1}^2 (\F_i,\V_i)_{\Omega_i}.
\end{equation}
The variational problem now reads as follows:
\begin{prob}[Weak formulation]\label{weakform}
    Find $(\U,\lambda) \in \bV \times \varLambda$ such that
    \begin{equation}
        \Bf(\U,\lambda; \V,\eta-\lambda) \leq \Lf(\V) \quad \forall (\V,\eta) \in \bV \times \varLambda.
        \label{weakform}
    \end{equation}
\end{prob}
We refer to \cite{HHNL88,HHN96} for the derivation of weak formulation from Problem~\ref{prob:strong} and for the proof of existence and uniqueness of solutions to 
problem \eqref{weakform}.

\section{Finite element method}

Let the bodies $\Omega_i \subset \mathbb{R}^d$ be separately divided into sets of
non-overlapping simplices $\Ch^i$, $i=1,2$.  The $d-1$ dimensional facets of
the elements in $\Ch^i$ are further divided into the set of interior facets $\Eh^i$,
the set of facets on the contact boundary $\Gh^i$, and the set of facets on the Neumann boundary $\Nh^i$.
We denote by $\Gcap$ the boundary mesh on $\Gamma$
which is obtained by intersecting the facets of $\Gh^1$ and $\Gh^2$.
In particular, each $E \in \Gcap$ corresponds to a pair $(E_1, E_2) \in \Gh^1
\times \Gh^2$ such that $E = E_1 \cap E_2$.
The finite element subspaces are
\begin{align}
    \bV_{i,h} &= \{ \V_{i,h} \in \bV_i : \V_{i,h}|_K \in [P_p(K)]^d~\forall K \in \Ch^i  \}, \\
    \bV_h &= \bV_{1,h} \times \bV_{2,h},\\
    Q_h &= \{ \eta_h \in \mQ : \eta_h|_E \in P_p(E)~\forall E \in \Gcap  \},
\end{align}
where $P_p(K)$ denotes the polynomials of degree $p$ on $K$.
Moreover, we introduce a subset of $\varLambda$,  denoted by $\varLambda_h $, as the positive part of $Q_h$, i.e.
\begin{equation}
    \varLambda_h = \{ \eta_h \in Q_h : \eta_h \geq 0 \} .
\end{equation}
Now, defining a stabilised bilinear form $\Bf_h$ through
\begin{equation}
    \Bf_h(\W_h,\xi_h;\V_h,\eta_h) = \Bf(\W_h,\xi_h;\V_h,\eta_h) - \alpha \Sf(\W_h,\xi_h;\V_h,\eta_h),
\end{equation}
where $\alpha > 0$ is a stabilisation parameter and
\begin{equation}
    \Sf(\W_h,\xi_h;\V_h,\eta_h) = \sum_{i=1}^2 \sum_{E \in \Gh^i} \frac{h_E}{\mu_i}\Big(\xi_h + \sihnw, \eta_h + \sihnv\Big)_E,
    \label{stabterm}
\end{equation}
we arrive at the following finite element formulation which is an extension of the mortar method introduced in 
\cite{JS2012,GSVmortar}.

\begin{prob}[Stabilised discrete formulation]
    Find $(\U_h,\lambda_h) \in \bV_h \times \varLambda_h$ such that
    \begin{equation}\label{VIh}
        \Bf_h(\U_h,\lambda_h; \V_h,\eta_h-\lambda_h) \leq \Lf(\V_h) \quad \forall (\V_h,\eta_h) \in \bV_h \times \varLambda_h.
    \end{equation}
    \label{prob:stab}
\end{prob}

We will now derive an equivalent formulation wherein the Lagrange multiplier is not explicitly present.
To this end, we start by defining $L^2(\Gamma)$-functions $\h_i$  through
\begin{equation}
\h_i|_E = h_E\quad \forall E \in \Gh^i, \ i=1,2,
\end{equation}
 and introduce the  notation
\begin{equation}
    \mean{\sigma_n(\U_h)} = \frac{\h_1\mu_2}{\h_1\mu_2 + \h_2\mu_1}\, \sonehnu + \frac{\h_2\mu_1}{\h_1\mu_2 + \h_2\mu_1}\,\stwohnu,
\end{equation}
i.e.  a convex combination of the discrete normal tractions. Furthermore, we let
\begin{equation}\label{lh}
l_h(\U_h)= -  \mean{\sigma_n(\U_h)}  - \beta_h \jump{\uhn} ,
\end{equation}
 where 
\begin{equation}
    \beta_h = \frac{\mu_1 \mu_2}{\alpha(\h_1\mu_2 + \h_2 \mu_1)}.
\end{equation}
Next, we will show that the discrete Lagrange multiplier $\lambda_h$ can be eliminated  locally (i.e. element by element). This leads to a Nitsche formulation with the displacements as sole unknowns.
Choosing $\V_h=\Z$ in the variational inequality \eqref{VIh},
 gives
\begin{equation}
  - \dual{\jump{\uhn} ,\eta_h-\lambda_h }
-\alpha  \sum_{i=1}^2 \sum_{E \in \Gh^i} \frac{h_E}{\mu_i}\big(\lambda_h + \sihnu, \eta_h -\lambda_h)_E \leq 0,
\end{equation}
which, in view of the notation defined above, can be written as
  \begin{equation}\label{discvi}
    \langle \lambda_h-l_h(\U_h), \eta_h-\lambda_h\rangle \leq 0\quad \forall \eta_h \in \Lambda_h.
    \end{equation}
   Let then $E\in  \Gcap$ be an element on which $\lambda_h\vert_E >0$ and 
 denote by  $\phi_E$ one of the basis functions of $Q_h \vert _E$. Moreover,  
   choose a test function $\eta_h$ in \eqref{discvi}  in such a way that it vanishes at $\Gamma\setminus E$ and  $\eta_h\vert _E =\lambda_h \pm  \epsilon \phi_E$, with $\epsilon >0$ chosen small enough so that $\eta_h\vert _E>0$. 
It follows that 
      \begin{equation}
   0= \langle \lambda_h-l_h(\U_h), \phi_E \rangle=  \int_E  \big(\lambda_h-l_h(\U_h)\big) \phi_E\, ds \end{equation}
and, since
   \begin{equation}
  \big( \lambda_h-l_h(\U_h)\big)\vert_E \in Q_h\vert_E,
   \end{equation}
   we conclude that 
    \begin{equation} \big( \lambda_h-l_h(\U_h)\big)\vert_E =0. 
    \end{equation}
    This shows that 
      \begin{equation} 
    \lambda_h =(l_h(\U_h))_+ \, ,
    \label{eq:lambdah}
    \end{equation}
  where
  $(a)_+ =
\max(0,a)$ denotes the positive part of $a$.

The discrete contact region, defined as
\begin{equation}
   \Gamma_c(\U_h)= \{\,\X\in \Gamma  :  \lambda_h(\X)>0 \, \},
\end{equation}
can now, in view of \eqref{eq:lambdah}, be written as 
\begin{equation}
    \Gamma_c(\U_h) = \{ \X \in \Gamma : l_h(\U_h(\X)) > 0 \}.
    \label{contactregion}
\end{equation}
On the other hand, testing with $\V_h$ in \eqref{VIh} and using \eqref{eq:lambdah} yields
\begin{equation}
\begin{aligned}
\sum_{i=1}^2 &(\Sig_i(\U_{i,h}),\Eps(\V_{i,h}))_{\Omega_i} - \dual{\jump{\vhn}  , (l_h(\U_h))_+ } 
\\
&
-\alpha \sum_{i=1}^2 \sum_{E \in \Gh^i} \frac{h_E}{\mu_i}\Big((l_h(\U_h))_+  + \sihnu, \ \sihnv\Big)_E
\\&= \sum_{i=1}^2 (\F_i,\V_{i,h})_{\Omega_i} \quad \forall \V_h \in \bV_h.
\end{aligned}
\label{nsystem}
\end{equation}
It follows from  \eqref{lh}  that 
\begin{equation}
\begin{aligned}
-&\dual{\jump{\vhn}  , (l_h(\U_h))_+ } 
\\
&=\Big(   \mean{\sigma_n(\U_h)}  ,  \jump{\vhn}   \Big)_{  \Gamma_c(\U_h)}+
\Big(     \beta_h \jump{\uhn}  , \jump{\vhn}    \Big)_{  \Gamma_c(\U_h)},
\end{aligned}
\end{equation}
and on $      \Gamma_c(\U_h) $ it holds that
\begin{align}
(l_h(\U_h))_+  +\sigma_{1,n}(\U_1)&= \frac{\h_2\mu_1}{\h_1\mu_2 + \h_2\mu_1}\big( \sigma_{1,n}(\U_1)- \sigma_{2,n}(\U_2)\big)   - \beta_h \jump{\uhn}, 
\\
(l_h(\U_h))_+  +\sigma_{2,n}(\U_2)&= \frac{\h_1\mu_2}{\h_1\mu_2 + \h_2\mu_1}\big( \sigma_{2,n}(\U_2)- \sigma_{1,n}(\U_2)\big)   - \beta_h \jump{\uhn}. 
\end{align}
Therefore, defining the jump
\begin{equation}
\jump{\sigma_n(\U_h)}=   \sigma_{2,n}(\U_1)- \sigma_{1,n}(\U_2),
\end{equation}
and the  $L^2(\Gamma)$-function
\begin{equation}
\gamma_h=\frac{\alpha \h_1\h_2}{\h_1 \mu_2 +\h_2 \mu_1},\end{equation}
 and substituting the above five expressions into \eqref{nsystem}, we obtain  after rearranging terms the following Nitsche's formulation for Problem \ref{prob:stab} with $\U_h$ as the sole unknown.
\begin{nitsche} 
    Find $\U_h \in \bV_h$ such that
   \begin{equation}
    \begin{aligned}
        &\sum_{i=1}^2 (\sigma_i(\U_{i,h}),\Eps(\V_{i,h}))_{\Omega_i} +\Big(     \beta_h \jump{\uhn}  , \jump{\vhn}    \Big)_{  \Gamma_c(\U_h)}\\
       & +\Big(   \mean{\sigma _n(\U_h)}  ,  \jump{\vhn}   \Big)_{  \Gamma_c(\U_h)}+  \Big(   \mean{\sigma_n(\V_h)}  ,  \jump{\uhn}   \Big)_{  \Gamma_c(\U_h)}   
       \\
       &-  \Big(  \gamma_h  \jump{\sigma_n(\U_h)}  ,   \jump{\sigma_n(\V_h)}  \Big)_{  \Gamma_c(\U_h)}  
       \\&
     -\alpha \sum_{i=1}^2 \Big(\frac{\h_i}{\mu_i}  \sihnu, \ \sihnv\Big) _{\Gamma \setminus  \Gamma_c(\U_h)}
        \\
        & = \sum_{i=1}^2 (\F_i,\V_{i,h})_{\Omega_i} \quad \forall \V_h \in \bV_h.
    \end{aligned}
    \end{equation}
\end{nitsche}
\begin{rem} Since $\sinu$ vanishes on $ \Gamma \setminus  \Gamma_c(\U_h) $, this set can be reinterpreted as being part of $\Gamma_{N,i}$, $i=1,2$. Consequently, 
the term $$ \alpha \sum_{i=1}^2 \Big(\frac{\h_i}{\mu_i}  \sihnu, \ \sihnv\Big) _{\Gamma \setminus  \Gamma_c(\U_h)}$$ can be dropped.  
\label{drop}
\end{rem}

Next we present two other variants of Nitsche's method. The first is the so called "master-slave" formulation. 

Assume that the material parameters satisfy $\mu_1\geq \mu_2$. The body $\Omega_1$ is the master part, $\Omega_2$ the slave, and the mortaring at the contact surface is only done for the latter, less rigid body, i.e. the stabilising   term  is now 
\begin{equation}
    \Sf(\W_h,\xi_h;\V_h,\eta_h) =   \sum_{E \in \Gh^2} \frac{h_E}{\mu_2}\Big(\xi_h + \sigma_{2,n}(\W_{2,h}), \eta_h + \sigma_{2,n}(\V_{2,h})\Big)_E.
\end{equation}
Repeating the steps  above, we obtain
$
    \lambda_h =(l_h(\U_h))_+ \, ,
$  
with
\begin{equation}\label{lh2}
l_h(\U_h)= -  \sigma_{2,n}(\U_{2,h}) - \frac{\mu_2}{\alpha \h_2} \jump{\uhn} .
\end{equation}
The contact region $ \Gamma_c(\U_h) $ is given by \eqref{contactregion}, with $l_h(\U_h)$ taken from \eqref{lh2},
 and we have the following method.
\begin{nitsche} 
    Find $\U_h \in \bV_h$ such that
   \begin{equation}
    \begin{aligned}
        &\sum_{i=1}^2 (\Sig_i(\U_{i,h}),\Eps(\V_{i,h}))_{\Omega_i} +\Big(     \frac{\mu_2}{\alpha \h_2} \jump{\uhn}  , \jump{\vhn}    \Big)_{  \Gamma_c(\U_h)
        }\\
       & +\Big(    \sigma_{2,n}(\U_{2,h})  ,  \jump{\vhn}   \Big)_{  \Gamma_c(\U_h)}+  \Big(  \sigma_{2,n}(\V_{2,h})  ,  \jump{\uhn}   \Big)_{  \Gamma_c(\U_h)}   
       \\
       &
     -\alpha   \Big(\frac{\h_2}{\mu_2}  \sigma_{2,n}(\U_{2,h}),  \sigma_{2,n}(\V_{2,h}) \Big) _{\Gamma \setminus  \Gamma_c(\U_h)}
        \\
        &\qquad = \sum_{i=1}^2 (\F_i,\V_{i,h})_{\Omega_i} \quad \forall \V_h \in \bV_h.
    \end{aligned}
    \end{equation}
        \end{nitsche}
 
 \noindent Again, the term
    $$
\alpha   \Big(\frac{\h_2}{\mu_2}  \sigma_{2,n}(\U_{2,h}),  \sigma_{2,n}(\V_{2,h}) \Big) _{\Gamma \setminus  \Gamma_c(\U_h)}
    $$
    can be dropped, see Remark \ref{drop}.

In the third alternative, we follow \cite{Juntunen2015} and define the stabilising term through
\begin{equation}
  \alpha   \Sf(\W_h,\xi_h;\V_h,\eta_h) =  \Big(\beta_h^{-1}(\xi_h +  \mean{\sigma _n(\W_h)} ), \eta_h +  \mean{\sigma _n(\V_h)} \Big)_\Gamma.
\end{equation}
Repeating once more the above computations,  we arrive at the following method.
\begin{nitsche} 
    Find $\U_h \in \bV_h$ such that
   \begin{equation}
    \begin{aligned}
        &\sum_{i=1}^2 (\Sig_i(\U_{i,h}),\Eps(\V_{i,h}))_{\Omega_i} +\Big(     \beta_h \jump{\uhn}  , \jump{\vhn}    \Big)_{  \Gamma_c(\U_h)}\\
       & +\Big(   \mean{\sigma_n(\U_h)}  ,  \jump{\vhn}   \Big)_{  \Gamma_c(\U_h)}+  \Big(   \mean{\sigma_n(\V_h)}  ,  \jump{\uhn}   \Big)_{  \Gamma_c(\U_h)}   
       \\
       &
 -   \Big(\beta_h^{-1}( \mean{\sigma _n(\U_h)} ),    \mean{\sigma _n(\V_h)} \Big) _{\Gamma \setminus  \Gamma_c(\U_h)}
            \\
        &\ = \sum_{i=1}^2 (\F_i,\V_{i,h})_{\Omega_i} \quad \forall \V_h \in \bV_h,
    \end{aligned}
    \end{equation}
    with $\Gamma_c(\U_h)$ given by \eqref{contactregion} $($and $l_h(\U_h)$ as in   \eqref{eq:lambdah}$)$.
      \end{nitsche}
  
  \noindent  Also here the term 
    \begin{equation}
    \Big(\beta_h^{-1}( \mean{\sigma _n(\U_h)} ),    \mean{\sigma _n(\V_h)} \Big) _{\Gamma \setminus  \Gamma_c(\U_h)}
    \label{dropterm}
    \end{equation}
    can be dropped.

\section{Error analysis}

The energy norm for the problem is
\begin{equation}
    \ \sum_{i=1}^2 (\Sig_i(\W_i), \Eps(\W_i))_{\Omega_i}.
    \end{equation}
  Since we  exclude nearly incompressible materials,  it holds $\lambda_i \lesssim \mu_i$, and hence 
with our choice of  boundary conditions  the Korn
   inequality is valid in both regions, and  we have the norm equivalence
   \begin{equation}
    \sum_{i=1}^2 (\Sig_i(\W_i), \Eps(\W_i))_{\Omega_i} \approx \sum_{i=1}^2\mu_i \Vert \W \Vert_{1,\Omega_i} ^2.
    \label{normequi}
    \end{equation}
    The error estimate will be given in the continuous norm
    \begin{equation}
    \enorm{(\W,\xi)}^2  = \sum_{i=1}^2\Big( \mu_i \Vert \W \Vert_{1,\Omega_i} ^2+ \frac{1}{\mu_i} \| \xi \|_{-\frac12, \Gamma}^2\Big)
    \end{equation} 
  but in the analysis we will also use the following mesh dependent norm
    \begin{equation}
    \enorm{(\W_h,\xi_h)}_h^2  = \enorm{(\W_h,\xi_h)}^2 + \sum_{i=1}^2 \sum_{E \in \Gh^i} \frac{h_E}{\mu_i} \| \xi_h \|_{0,E}^2.
\end{equation}
\begin{thm}[Continuous stability]
    For every $(\W,\xi) \in \bV \times Q$ there exists $\V \in \bV$ such that
    \begin{equation}
        \Bf(\W,\xi;\V,-\xi) \gtrsim \enorm{(\W,\xi)}^2
    \end{equation}
    and
    \begin{equation}
        \|\V\|_V \lesssim \enorm{(\W,\xi)}.
    \end{equation}
\end{thm}
\begin{proof}
It is well-known that the inf-sup condition 
\begin{equation}
\sup_{\bZ_i\in \bV_i} \frac{\langle - \bZ_{i}\cdot\N_i, \xi\rangle}{\Vert \nabla \bZ_i\Vert_{0,\Omega_i}}
\geq C_i \Vert   \xi   \Vert _{-\frac12, \Gamma} \qquad \forall \xi\in Q,
\label{auxinfsup}
\end{equation}
holds in  both subdomains $\Omega_i$ (cf. \cite{Babuska1973}).
Therefore \begin{equation}\label{continfsup}
\sup_{\bZ=(\bZ_1,\bZ_2)\in \bV} \frac{\langle \jump{z_n}, \xi\rangle}{(\sum_{i = 1}^2 \mu_i \|\nabla \bZ_i\|_{0,\Omega_i}^2)^{1/2}}
\geq C   \left(\frac{1}{\mu_1} + \frac{1}{\mu_2} \right)^{1/2} \|\xi\|_{-\frac12,\Gamma} \qquad \forall \xi\in Q\, .
\end{equation}
Assume then that $(\W,\xi)\in\bV \times  Q$ is given  and let $\V_i=\W_i-\bQ_i$  where $\bQ_i\in\bV_i$ solves the problem
\[
(\Sig_i(\bQ_i), \Eps(\bZ_i))_{\Omega_i} = \langle -\bZ_i\cdot\N_i,\xi\rangle\quad \forall \bZ_i\in \bV_i\, , \ \ i=1,2\,.
\]
Choosing $\bZ_i=\bQ_i$ above, we obtain after summing 

\[
\sum_{i=1}^2(\Sig_i(\bQ_i), \Eps(\bQ_i))_{\Omega_i} = \langle \llbracket q_n\rrbracket,\xi\rangle\, .
\]
Moreover, from \eqref{auxinfsup}, it follows that
\[
\Vert   \xi   \Vert _{-\frac12, \Gamma} \lesssim \sup_{\bZ_i\in \bV_i} \frac{\langle - \bZ_{i}\cdot\N_i, \xi\rangle}{\Vert \nabla \bZ_i\Vert_{0,\Omega_i}} =
\sup_{\bZ_i\in \bV_i} \frac{(\Sig_i(\bQ_i), \Eps(\bZ_i))_{\Omega_i}}{\Vert \nabla \bZ_i\Vert_{0,\Omega_i}} \lesssim  \mu_i\|\bQ_i\|_{1,\Omega_i}
\]
and thus 
\[
\left(\frac{1}{\mu_1} + \frac{1}{\mu_2} \right)^{1/2} \|\xi\|_{-\frac12,\Gamma} \lesssim \Big( \sum_{i=1}^2 \mu_i \Vert \bQ_i \Vert_{1,\Omega_i} ^2\Big)^{1/2} 
 \, .
\]
Now, it is easy to see that
\begin{align*}
       \Bf(\W,\xi;\V,-\xi) &= \sum_{i=1}^2 \Big\{(\Sig_i(\W_i), \Eps(\W_i))_{\Omega_i}-  (\Sig_i(\W_i), \Eps(\bQ_i))_{\Omega_i}\Big\}+\langle \llbracket q_n\rrbracket,\xi\rangle\\
     &  \gtrsim \sum_{i=1}^2 \mu_i \Vert \W_i \Vert_{1,\Omega_i}^2 -\frac{1}{2}\sum_{i=1}^2 \mu_i \Vert \W_i \Vert_{1,\Omega_i}^2 - \frac{1}{2}\sum_{i=1}^2 \mu_i \Vert \bQ_i \Vert_{1,\Omega_i}^2    \\
 & \quad  +   \sum_{i=1}^2 (\Sig_i(\bQ_i), \Eps(\bQ_i))_{\Omega_i} \\
   &  \gtrsim   \sum_{i=1}^2 \mu_i \Vert \W_i \Vert_{1,\Omega_i}^2 +
     \left(\frac{1}{\mu_1} + \frac{1}{\mu_2} \right)\|\xi\|_{-\frac12,\Gamma}^2 = \enorm{(\W,\xi)}^2
       \end{align*}
       and that $        \|\V\|_V = \|\W-\bQ\|_V \lesssim \enorm{(\W,\xi)}.$
\end{proof}

\noindent
Above and in the following we write $a\gtrsim b$ (or $a \lesssim b$) when $a \geq Cb$ (or $a\leq Cb$) for some positive constant $C$ independent of the finite element mesh.

To derive the discrete stability estimate, we need the following discrete trace inequality, easily shown by a scaling argument.
\begin{lem}[Discrete trace estimate]  
    There exists $C_I > 0$, independent of the mesh parameter $h$, such that
    \begin{equation}
        C_I \sum_{E \in \Gh^i} \frac{h_E}{\mu_i} \| \sihnv \|_{0,E}^2 \leq \mu_i\| \V_{i,h}\|_{1,\Omega_i}^2 \quad \forall \V_{i,h} \in \bV_i, \quad i = 1,2.
            \label{dtrace}
    \end{equation}
\end{lem}
\begin{thm}[Discrete stability] Suppose that $0<\alpha < C_I$. Then,
    for every $(\W_h,\xi_h) \in \bV_h \times Q_h$, there exists $\V_h \in \bV_h$ such that
    \begin{equation}
        \Bf_h(\W_h,\xi_h;\V_h,-\xi_h) \gtrsim \enorm{(\W_h,\xi_h)}_h^2
    \end{equation}
    and
    \begin{equation}
        \|\V_h\|_V \lesssim \enorm{(\W_h,\xi_h)}_h.
    \end{equation}
\end{thm}
\begin{proof}
From the  discrete trace estimate it follows that
\begin{align*}
 \Bf_h(\W_h,\xi_h;\W_h,-\xi_h) \geq  \left(1-\frac{\alpha}{C_I}\right)\, \sum_{i=1}^2 \mu_i \Vert \W_{i,h} \Vert_{1,\Omega_i}^2 +   \alpha  \sum_{i=1}^2 \sum_{E\in\Gh}\frac{h_E}{\mu_i} \|\xi_h\|_{0,E}^2 \,,
\end{align*}
which   proves the result in the mesh-dependent norm of $\xi_h$ for $0<\alpha < C_I$.

On the other hand,  the continuous inf-sup condition \eqref{continfsup} implies that for any $\xi_h\in Q_h$ there exists $\V\in\bV$ such that
\[
 \frac{\langle \jump{v_n}, \xi_h\rangle}{(\sum_{i = 1}^2 \mu_i \|\nabla \V_i\|_{0,\Omega_i}^2)^{1/2}}
\geq C_1   \left(\frac{1}{\mu_1} + \frac{1}{\mu_2} \right)^{1/2} \|\xi_h\|_{-\frac12,\Gamma} \, .
\]
This means that (cf. the proof of Lemma 3.2 in \cite{GSV17})
\begin{align}
\label{prop1} \langle \jump{(I_hv)_n}, \xi_h\rangle & \geq C_2      \left(\frac{1}{\mu_1} + \frac{1}{\mu_2} \right)\|\xi_h\|_{-\frac12,\Gamma}^2 
-C_3 \sum_{i=1}^2 \sum_{E\in\Gh}\frac{h_E}{\mu_i} \|\xi_h\|_{0,E}^2 \\
\label{prop2}
\sum_{i = 1}^2 \mu_i \| I_h\V_i\|_{1,\Omega_i} & \leq C_4      \left(\frac{1}{\mu_1} + \frac{1}{\mu_2} \right)\|\xi_h\|_{-\frac12,\Gamma}^2
\end{align}
where $C_2,C_3,C_4$ are positive constants and  $I_{h}\V\in\bV_h$ is the Cl\'ement interpolant of $\V$.  Using again the  discrete trace estimate and  inequalities \eqref{prop1} and \eqref{prop2}, we then obtain
\begin{align*}
 \Bf_h(\W_h,\xi_h;-I_h\V,0) =& - \sum_{i=1}^2 (\Sig_i(\W_{i,h}), \Eps(I_h\V_i))_{\Omega_i} +\langle \jump{(I_hv)_n}, \xi_h\rangle \\
 & 
 - \sum_{i=1}^2 \sum_{E \in \Gh^i} \frac{h_E}{\mu_i}\Big(\xi_h + \sihnw, \sigma_{i,n}(I_h\V_i)\Big)_E, \\
 \geq & \ C_5      \left(\frac{1}{\mu_1} + \frac{1}{\mu_2} \right)\|\xi_h\|_{-\frac12,\Gamma}^2  -C_6 \sum_{i = 1}^2 \mu_i \|\W_{i,h}\|_{1,\Omega_i}  \\
 & -C_7 \sum_{i=1}^2 \sum_{E\in\Gh}\frac{h_E}{\mu_i} \|\xi_h\|_{0,E}^2 .
\end{align*}

Now, it is straightforward to show (cf. \cite{GSVmortar}) that there exists $\delta >0$ such that
\[
 \Bf_h(\W_h,\xi_h;\W_h-\delta I_h\V,-\xi_h)  \gtrsim \enorm{(\W_h,\xi_h)}_h^2\,.
 \]
 and that $\|\W_h-\delta I_h\V\|_V\lesssim \enorm{(\W_h,\xi_h)}_h$.
\end{proof}
In our improved error analysis, we use techniques from the a posteriori error analysis. 
Let $\F_{i,h}\in \bV_{i,h}$ be the $[L^2(\Omega_i)]^d$ projection of $\F_i$, define on any $K\in \Ch^i$ the  oscillation of $\F_i$ by
\[
{\rm osc}_K (\F_i )= h_K \|\F_i-\F_{i,h}\|_{0,K}\, , \qquad i=1,2,
\]
and, for each $E \in \Gh^i$, let $K(E) \in \Gh^i$ denote the element such that $\partial K(E)\cap E =E$.
\begin{lem}
For any $(\V_h,\eta_h)\in \bV_h\times Q_h$, it holds that
    \begin{equation}\begin{aligned}
  \Big(    \sum_{i=1}^2 \sum_{E \in \Gh^i}  & \frac{h_E}{\mu_i} \left\| \eta_h + \sihnv \right\|_{0,E}^2 \Big)^{1/2} \\ & \leq    \enorm{(\U-\V_h,\lambda-\eta_h)} + 
    \Big(    \sum_{i=1}^2 \mu_i^{-1}\sum_{E \in \Gh^i} {\rm osc}_{K(E)} (\F_i)^2\Big)^{1/2}.
    \end{aligned}
    \end{equation}
    \label{osclem}
\end{lem}
\begin{proof} We  follow the reasoning presented for  the mortar method in \cite{GSVmortar}. It is clearly enough to prove the result in  $\Omega_1$. Thus, let  $b_E \in P_d(E)$, $E \in \Gh^1$, be the usual edge/facet bubble function
 and   define $\tau_E$  on $K(E) \in \Ch^1$ through
    \[
        \tau_E\big|_E = \frac{h_E b_E}{\mu_1} \Big(\eta_h  + \sonehnv \Big) \quad \text{and} \quad \tau_E\big|_{\partial K(E) \setminus E} = 0,
    \]
    where $K(E) $  is  such that $\overline{K(E)} \cap E = E$.
    It follows that
    \begin{equation}
        \frac{h_E}{\mu_1}\Big\|\eta_h+ \sonehnv \Big\|^2_{0,E} \lesssim 
                                                                \Big( \eta_h+\sonehnv , \tau_E \Big)_E.
  \label{aux1}
  \end{equation}
  Next, defining $\Tau \in \bV_{1,h}$ in such a way that $\tau_n:=-\Tau\cdot\N=  \sum_{E \in \Gh^1} \tau_E$ and testing problem \eqref{weakform} with $(\V_1, \V_2, \eta) = (-\Tau, 0, \lambda)$, we obtain 
  \[
   0\leq ( \Sig_1(\U_1),\Eps(\Tau))_{\Omega_1} - \dual{\tau_n, \lambda}-(\F_1,\Tau)_{\Omega_1}.
   \]
    Summing \eqref{aux1}  over the edges in $\Gh^1$, gives then
    \begin{align*}
        &\sum_{E \in \Gh^1} \frac{h_E}{\mu_1} \Big\|\eta_h+ \sonehnv\Big\|^2_{0,E} \\
        &\lesssim \langle \tau_n, \eta_h - \lambda \rangle +      (\Sig_1(\U_1),\Eps(\Tau))_{\Omega_1}   -  (\F_1,\Tau)_{\Omega_1}    
   + \sum_{E \in \Gh^1}  (\sonehnv ,  \tau_E)_E  \\
 &  =   \langle \tau_n, \eta_h - \lambda \rangle +      (\Sig_1(\U_1),\Eps(\Tau))_{\Omega_1}   -  (\F_1,\Tau)_{\Omega_1}  \\
 &\quad   -
( \Div \Sig_1(\V_{1,h}),\Tau)_{\Omega_1} -  (\Sig_1(\V_{1,h}) , \Eps(\Tau)_{\Omega_1}
   \\
        &= \langle \tau_n, \eta_h - \lambda \rangle +      (\Sig_1(\U_1)-\Sig_1(\V_{1,h}),\Eps(\Tau))_{\Omega_1}   -  ( \Div \Sig_1(\V_{1,h})+\F_1,\Tau)_{\Omega_1} \, .
    \end{align*}
Inverse estimates imply that
\begin{equation}
\mu_1\|\Tau\|_{1,\Omega_1}^2 \lesssim \mu_1 \sum_{E \in \Gh^1}  h_E^{-2} \|\tau_E\|_{0,K(E)}^2 \, \lesssim \sum_{E \in \Gh^1}   \frac{h_E}{\mu_1} \left\| \eta_h + \sonehnv \right\|_{0,E}^2 \,.
\label{invest}
\end{equation}
Now, one readily sees, using trace inequalities and the norm equivalence \eqref{normequi}, that
\begin{align*}
& \sum_{E \in \Gh^1} \frac{h_E}{\mu_1} \Big\|\eta_h+\sonehnv\Big\|^2_{0,E}  \\ & \quad \lesssim
\mu_1^{-1/2}\|\eta_h - \lambda||_{-\frac{1}{2},\Gamma} \, \mu_1^{1/2} \|\Tau\|_{1,\Omega_1}  +
 \mu_1^{1/2} \|\U_1-\V_{1,h}\|_{1,\Omega_1}\, \mu_1^{1/2} \|\Tau\|_{1,\Omega_1}\\
 &\quad + \bigg(\sum_{E \in \Gh^1} \frac{h_E^2}{\mu_1} \|\Div \Sig_1(\V_{1,h})+\F_1\|^2_{0,E}\bigg)^{1/2} \ \bigg( \mu_1 \sum_{E \in \Gh^1}  h_E^{-2} \|\tau_E\|_{0,K(E)}^2\bigg)^{1/2}\, ,
\end{align*}
from which, using the standard estimates for interior residuals (cf. \cite{Vbook}) and the inverse estimate \eqref{invest} to bound the last term,  it follows that
\begin{align*}
& \big(\sum_{E \in \Gh^1} \frac{h_E}{\mu_1} \Big\|\eta_h+\sonehnv\Big\|^2_{0,E} \Big)^{1/2} \lesssim 
 \enorm{(\U-\V_h,\lambda-\eta_h)}  + \Big(\mu_1^{-1}\sum_{E \in \Gh^1} {\rm osc}_{K(E)} (\F_1)^2\Big)^{1/2} \,.
\end{align*}

\end{proof}

We can now establish the quasi-optimality of the method.

\begin{thm} For $0<\alpha< C_I$ 
    it holds that
    \begin{equation}\begin{aligned}
        \enorm{(\U-\U_h, \lambda -   \lambda_h)}    \lesssim  &  \inf_{(\V_h,\eta_h)\in  \bV_h\times\Lambda_h}\Big(\enorm{(\U-\V_h,\lambda-\eta_h)}  +
           \sqrt{\langle \jump{u_n} ,\eta_h\rangle} \Big)
            \\   & \qquad +
          \Big(   \sum_{i=1}^2  \mu_i^{-1} \sum_{E \in \Gh^i} {\rm osc}_{K(E)} (\F_i)^2\Big)^{1/2}.
         \end{aligned}
         \end{equation}
\end{thm}
\begin{proof}
On account of the discrete stability estimate, there exists $\W_h\in \bV_h$ such that 
\begin{equation} \label{whbound}
 \|\W_h\|_V \lesssim \enorm{(\U_h-\V_h,\lambda_h-\eta_h)}_h,
\end{equation}
   and 
\begin{equation}
 \enorm{(\U_h-\V_h,\lambda_h-\eta_h)}_h^2 \lesssim \Bf_h(\U_h-\V_h,\lambda_h-\eta_h;\W_h,\eta_h-\lambda_h) .
 \end{equation}
 Using the bilinearity and \eqref{VIh}, we obtain
     \begin{equation}
 \begin{aligned}
  & \Bf_h(\U_h-\V_h,\lambda_h-\eta_h;\W_h,\eta_h-\lambda_h)   
   \\  &\ = \Bf_h(\U_h,\lambda_h;\W_h,\eta_h-\lambda_h) -\Bf_h(\V_h,\eta_h;\W_h,\eta_h-\lambda_h) 
   \\
   & \ \lesssim  \Lf(\W_h)  -\Bf_h(\V_h,\eta_h;\W_h,\eta_h-\lambda_h) 
   \\
   & \  = \  \Bf(\U-\V_h,\lambda-\eta_h;\W_h,\eta_h-\lambda_h)    +  \Lf(\W_h)  
  \\    &\qquad -  \Bf(\U,\lambda;\W_h,\eta_h-\lambda_h)+\alpha \Sf(\V_h,\eta_h;\W_h,\eta_h-\lambda_h).
  \end{aligned}
      \end{equation}
      The terms above can be estimated as follows. First, continuity of the bilinear form $\Bf$ and inequality \eqref{whbound} yield
      \begin{equation}
      \Bf(\U-\V_h,\lambda-\eta_h;\W_h,\eta_h-\lambda_h)\lesssim  \enorm{(\U-\V_h,\lambda-\eta_h)}\,  \enorm{(\U_h-\V_h,\lambda_h-\eta_h)} .
      \end{equation}
      Next, using the weak formulation \eqref{weakform} and the fact that $\jump{u_n}\geq 0$ and $\lambda_h\geq 0$, we obtain
         \begin{equation}  
      \Lf(\W_h)  
       -  \Bf(\U,\lambda;\W_h,\eta_h-\lambda_h)           =    \langle \jump{u_n},\eta_h-\lambda_h\rangle\leq
        \langle \jump{u_n},\eta_h\rangle .
      \end{equation}
      Finally, from the discrete trace estimate \eqref{dtrace} it follows that 
      \begin{equation}
       \begin{aligned}
    &  \alpha \Sf(\V_h,\eta_h;\W_h,\eta_h-\lambda_h)   \\
  &  \qquad\lesssim  \Big(  \sum_{i=1}^2 \sum_{E \in \Gh^i}   \frac{h_E}{\mu_i} \left\| \eta_h + \sihnu\right\|_{0,E}^2 \Big)^{1/2}   \enorm{(\U_h-\V_h,\lambda_h-\eta_h)}_h.
  \end{aligned}
      \end{equation}
   Using Lemma \ref{osclem},     and  collecting the above estimates,  we arrive at the asserted error estimate.
       \end{proof}
\begin{rem}  We refrain from giving an a priori error estimate assuming a regular solution. The reasons are twofold. Firstly, contact singularities are inevitable and essential in contact problems.
Secondly, to derive an a priori bound, one would need  to estimate the term
$ \sqrt{\langle \jump{u_n} ,\eta_h\rangle}$, with $\eta_h$ being the interpolant to $\lambda$. Besides, and perhaps most importantly, one of the main results of this paper is the  fact that we do not need to assume that the solution belongs to $H^s$, with $ s>3/2. $
\end{rem}
For the a posteriori error analysis,  we define the local estimators
\begin{alignat}{1}
    \label{apost1} \eta_K^2 &= \frac{h_K^2}{\mu_i} \|\hspace{0.2mm}\Div \Sig_i(\U_{i,h}) + \F_i \|_{0,K}^2, \quad K \in \Ch^i,
     \\
    \label{apost2} \eta_{E,\Omega}^2 &= \frac{h_E}{\mu_i} \left \| \jump{\Sig_i(\U_{i,h})\N} \right\|_{0,E}^2, \quad E \in \Eh^i, 
    \\
    \label{apost3} \eta_{E,\Gamma}^2 &= \frac{h_E}{\mu_i} \left\{\left\| \lambda_h + \sihnu \right\|_{0,E}^2+ 
    \|\hspace{0.2mm} \Sig_{i,t}(\U_{i,h})\|_{0,E}^2\right\}
    \\
                                     &\qquad+\frac{\mu_i}{h_E} \Vert (\jump{u_{h,n} })_{-}\Vert_{0,E}^2,\quad E \in \Gh^i, 
        \nonumber                             \\
    \label{apost4} \eta_{E,\Gamma_N}^2 &= \frac{h_E}{\mu_i} \left \|\hspace{0.2mm}\Sig_i(\U_{i,h})\N \right\|_{0,E}^2, \quad E \in \Nh^i,  
\end{alignat}
with $i=1,2$. The corresponding global estimator $\eta$ is then defined as
\begin{equation}
    \eta^2 = \sum_{i=1}^2 \Big\{\sum_{K \in \Ch^i} \eta_K^2 + \sum_{E \in \Eh^i} \eta_{E,\Omega}^2 + \sum_{E \in \Gh^i} \eta_{E,\Gamma}^2 + \sum_{E \in \Nh^i} \eta_{E,\Gamma_N}^2 \Big\}.
\end{equation}
In addition, we need an estimator $S$ defined only globally as
\begin{equation}
    S^2 = \big((\jump{u_{h,n}})_+,\lambda_h\big)_\Gamma.
\end{equation}

\begin{thm}[A posteriori error estimate]
    It holds that
    \begin{equation}
        \enorm{(\U-\U_h,\lambda-\lambda_h)} \lesssim \eta + S.
        \label{apostupper}
    \end{equation}
\end{thm}
\begin{proof} In view of the continuous stability estimate, there exists $\V\in \bV $, with  
\begin{equation}\label{normalising}
\|\V\|_V \lesssim  
 \enorm{(\U-\U_h,\lambda-\lambda_h)}, 
\end{equation} and
\begin{equation}
 \enorm{(\U-\U_h   ,\lambda-\lambda_h)}^2 \lesssim \ \Bf(\U-\U_h,\lambda-\lambda_h; \V,\lambda_h-\lambda) .
\end{equation}
Let $\tilde{\V}\in \bV_h$ be  the Cl\'ement interpolant of $\V$. From \eqref{VIh}, it follows that 
\begin{equation}
0\leq -\Bf(\U_h, \lambda_h; \tilde{\V}, 0)+\alpha\Sf(\U_h,\lambda_h,-\tilde{\V},0)- \Lf (\tilde{\V}).
\end{equation}
 Using the weak formulation \eqref{weakform}, this gives
\begin{equation}
\begin{aligned}
&   \Bf(\U-\U_h,\lambda-\lambda_h; \V,\lambda_h-\lambda) \\ 
 & \qquad \lesssim \    \Lf(\V-\tilde{\V}) -\Bf(\U_h,\lambda_h; \V-\tilde{\V},\lambda_h-\lambda) +\alpha\Sf(\U_h,\lambda_h,-\tilde{\V},0).
\end{aligned}
\end{equation}
 Integrating by parts, we obtain for the first two terms above
\begin{equation}
\begin{aligned}
 &\Lf (\V   -   \tilde{\V}) -\Bf(\U_h,\lambda_h; \V-\tilde{\V},\lambda_h-\lambda) \\
 &= \sum_{i=1}^2 \sum_{K\in\Ch^i} (\Div \Sig_i(\U_{i,h}) +\F_i,\V_i-\tilde{\V}_i)_K
 \\&\quad  - \sum_{i=1}^2 \sum_{E\in\Eh^i} (\jump{\Sig_i(\U_{i,h})\N},\V_i-\tilde{\V}_i)_E \\
& \quad  -  \sum_{i=1}^2 \sum_{E\in\Nh^i} (\Sig_i(\U_{i,h})\N,\V_i-\tilde{\V}_i)_E 
 - \sum_{i=1}^2 \sum_{E\in\Gh^i} ( \situ, (\V_{i,t}-\tilde{\V}_{i,t}))_E \\
& \quad  - 
\sum_{i=1}^2 \sum_{E\in\Gh^i} (\lambda_h + \sihnu,(\V_i-\tilde{\V}_i)\cdot\N)_E + \langle\jump{u_{h,n}},\lambda_h-\lambda\rangle,
\end{aligned}
\end{equation}
 Moreover, using an inverse inequality for the $H^{1/2}(\Gamma)$-norm (cf. \cite{cf}) we get 
 \begin{equation}
\begin{aligned}
& \langle\jump{u_{h,n}},\lambda_h-\lambda\rangle 
  \leq \big( (\jump{u_{h,n}})_+,\lambda_h\big)_\Gamma + \big\langle(\jump{u_{h,n}})_{-},\lambda_h-\lambda\big\rangle
  \\
& \lesssim  \big( (\jump{u_{h,n}})_+,\lambda_h\big)_\Gamma 
\\
   &\quad + \enorm{(\U-\U_h,\lambda-\lambda_h)} \,  \big( (\mu_1+\mu_2) \Vert (\jump{u_{h,n}})_{-} \Vert_{1/2, \Gamma}^2\big)^{1/2}
    \\
& \lesssim  \big( (\jump{u_{h,n}})_+,\lambda_h\big)_\Gamma 
\\
   &\quad + \enorm{(\U-\U_h,\lambda-\lambda_h)} \,  \big(\sum_{i=1}^2 \sum_{E\in\Gh^i}  \frac{\mu_i}{h_E} \Vert (\jump{u_{h,n}})_{-} \Vert_{0,E}^2\big)^{1/2}.
\end{aligned}
\end{equation}
Finally, using the discrete trace estimate \eqref{dtrace} and  the standard bounds for the Cl\'ement interpolant, and recalling \eqref{normalising}, 
 we obtain for the stabilising term  
\begin{equation}\begin{aligned}
| \Sf(&\U_h  ,\lambda_h,-\tilde{\V},0)|
\\& \lesssim \Big(\sum_{i=1}^2 \sum_{E\in\Gh^i}  \frac{h_E}{\mu_i} \left\| \lambda_h + \sihnu\right\|_{0,E}^2 \Big)^{1/2}  
 \enorm{(\U-\U_h,\lambda-\lambda_h)} .
 \end{aligned}
\end{equation}
Estimate \eqref{apostupper} follows from collecting the above bounds. 
\end{proof}
 
 The estimator $\eta$ bounds the error from below. For the proof of the following theorem we refer to \cite{GSV17}.
 \begin{thm}[A posteriori estimate -- efficiency]
    It holds that
    \begin{equation}
     \eta  \lesssim  
      \enorm{(\U-\U_h,\lambda-\lambda_h)} .
        \label{apostupper}
    \end{equation}
\end{thm}

The analysis of Methods 2 and 3 is analogous. In the a posteriori estimates the term $$ \sum_{i=1}^2 \sum_{E \in \Gh^i}  \frac{h_E}{\mu_i} \left\| \lambda_h + \sihnu \right\|_{0,E}^2,$$
is replaced by
\begin{equation}
 \sum_{E \in \Gh^2}  \frac{h_E}{\mu_2} \left\| \lambda_h +\sigma_{2,n}(\U_{2,h}) \right\|_{0,E}^2,
\end{equation}
and
\begin{equation}
\Vert \beta_h^{-1/2}(\lambda_h +  \mean{\sigma _n(\U_h)})\Vert_{0, \Gamma}^2,
\end{equation}
for Method 2 and 3, respectively.

 \section{Computational experiments}

All computations presented in this section were obtained using the Nitsche formulation 3  with the term \eqref{dropterm} dropped. Had we considered other formulations, the results would have been practically identical.  We also note that since the stabilized/Nitsche's method is variationally conforming (as a mortaring method) it passes the patch test of \cite{Laursen03}, p. 425. This was confirmed numerically up to machine accuracy. 

We consider the geometry given by
\begin{equation}
    \Omega_1 = [0.5,1.0] \times [0.25,0.75], \quad \Omega_2 = [1,1.6] \times [0,1],
\end{equation}
and define the boundary conditions on the following subsets:
\begin{align}
    \Gamma_{D,1} &= \{ (x,y) \in \partial \Omega_1 : x=0.5 \}, \quad \Gamma_{N,1} = \partial\Omega_1 \setminus (\Gamma_{D,1} \cup \Gamma),\\
    \Gamma_{D,2} &= \{ (x,y) \in \partial \Omega_2 : x=1.6 \}, \quad \Gamma_{N,2} = \partial\Omega_2 \setminus (\Gamma_{D,2} \cup \Gamma).
\end{align}
Thus, the geometry is the one given in Figure~\ref{fig:twobodynot}. A nonmatching
discretisation of the geometry is depicted in Figure~\ref{fig:meshblocks}.
Initially, the material parameters are $E_1 = E_2 = 1$ and $\nu_1 = \nu_2 = 0.3$
and the loading is
\begin{equation}
  \F_1 = (x-0.5,0), \quad \F_2 = (0,0).
\end{equation}
For this loading, the displacement is constrained on $\Gamma_{D,i}$, $i=1,2$,
only in the horizontal direction which minimizes the effect of the singularities
-- other than the ones related to the contact boundary -- on the rates of
convergence.  We consider both linear and quadratic elements, with
$\alpha=10^{-2}$ and $\alpha=10^{-3}$, respectively.

The adaptively refined meshes are shown in
Figure~\ref{fig:adaptseq} and \ref{fig:adaptseq1}, and the global error estimator $\eta+S$ is plotted as a
function of the number of  degrees-of-freedom $N$ in
Figure~\ref{fig:adaptconv}. Since $\eta+S$ is an upper bound for
the total error, the results suggest that the total error of the quadratic solution is
limited to $\mathcal{O}(N^{-0.5})$ when using uniform refinements and that
adaptivity successfully improves the order of the discretisation error to $\mathcal{O}(N^{-1})$.

Next we fix also the vertical displacement on $\Gamma_{D,i}$, $i=1,2$,
and consider the loading
\begin{equation}
    \F_1 = (0,-0.05), \quad \F_2 = (0,0),
\end{equation}
which causes the left block to bend slightly downwards and, as a consequence,
the active contact region is a nontrivial subset of $\Gamma$. The active contact
region is found via an iterative solution of the linearised problem,
cf.~\cite{GSV17}.  See Figure~\ref{fig:adaptseq2} and \ref{fig:adaptseq21} for
the final meshes and contact stresses, and Figure~\ref{fig:adaptconv2} for the
convergence rates.  We observe that the singularity at the upper corner of the
contact region is properly resolved by the adaptive meshing strategy and that
the convergence is similar albeit less idealised as in the first
example.

In Figure~\ref{fig:adaptconv2E}, we demostrate how the improved convergence
rates can be obtained for $P_2$ elements even if the value of the Young's
modulus changes significantly over the contact boundary. In
Figure~\ref{fig:adaptconv2alpha}, we demonstrate that the effect of the
stabilisation parameter is small in the asymptotic limit.
Finally, in Figure~\ref{fig:final}, we consider the loading
\begin{equation}
    \F_1 = (-\cos(4 \pi (y - 0.5)),0), \quad \F_2 = (0,0),
\end{equation}
which results in an active contact boundary consisting of two
disjoint parts and a perfectly symmetric contact stress.

\begin{figure}[h!]
    \includegraphics[width=0.49\textwidth]{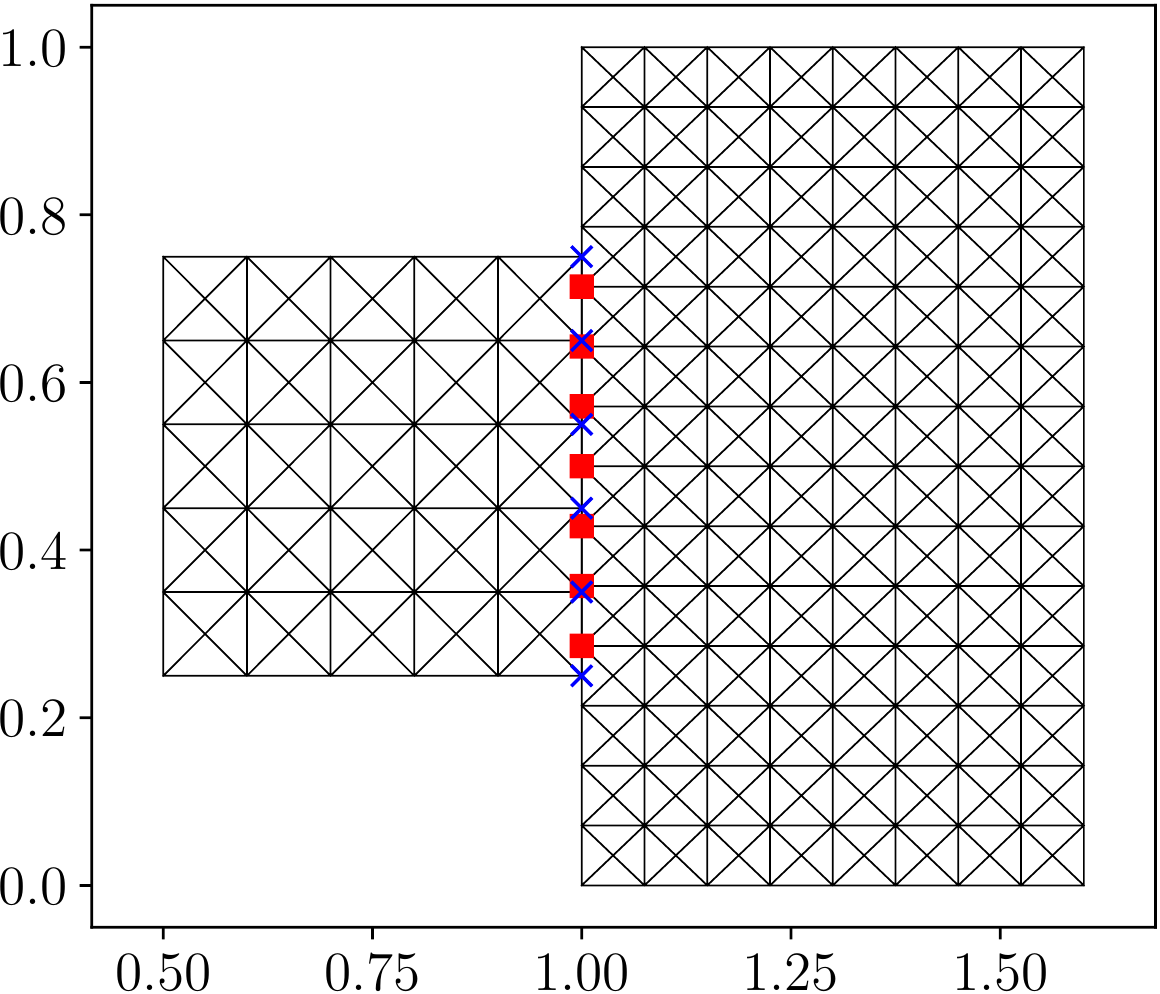}
    \caption{A finite element mesh and the vertices belonging to $\Gamma$.}
    \label{fig:meshblocks}
\end{figure}

\pgfplotstableread{
  ndofs eta
82 0.03973469149724339
272 0.024543634683048224
988 0.014696965226563396
3764 0.008649341647060958
14692 0.00486546952937192
}\unifPI

\pgfplotstableread{
  ndofs eta
82 0.03973469149724339
126 0.030035083325741986
200 0.023251855601360224
368 0.01718013761940071
628 0.01400852734514739
812 0.012194409588654358
1334 0.009715293148709643
1832 0.008303396548823559
2668 0.006941143911928724
3768 0.00600831970150925
5402 0.005077372254253719
7760 0.004241315176070536
10652 0.003648926885564471
14662 0.003162021500200948
}\adaptPI

\pgfplotstableread{
  ndofs eta
272 0.01614971558844067
988 0.008096327024555088
3764 0.003960091852532299
14692 0.0022786930407713165
}\unifPII

\pgfplotstableread{
  ndofs eta
272 0.01614971558844067
434 0.008290362778716209
662 0.004981300661617096
958 0.004050570185360159
1598 0.0025569231391579374
2130 0.0019464206297086092
2974 0.0015018117077493622
4364 0.0010527743099234864
5324 0.0008533196287345412
6728 0.0006541224049658154
8756 0.0004997496554225583
9824 0.0004514279641733795
12756 0.0003553917245163292
}\adaptPII

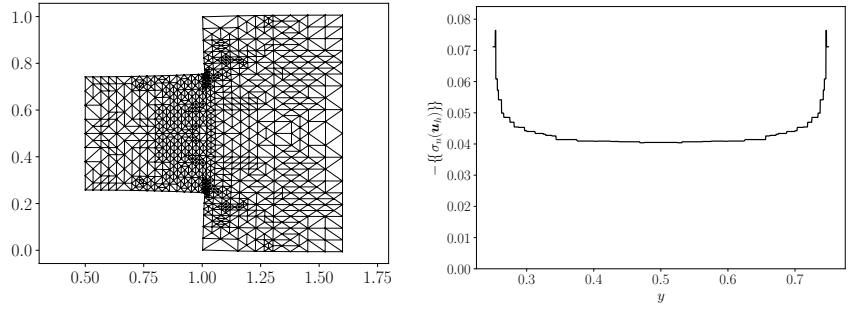
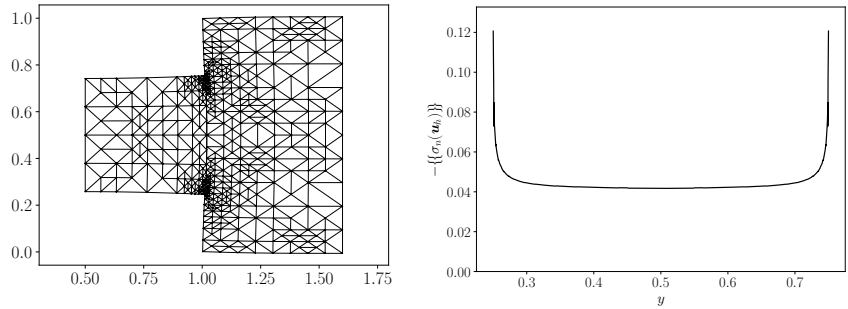
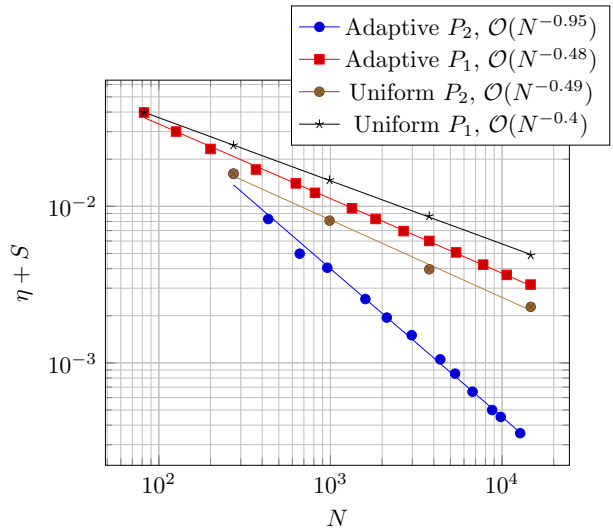
\begin{figure}[h!]
    \begin{subfigure}[t]{\textwidth}
        \centering
        \raisebox{-0.5\height}{\includegraphics[width=0.4\textwidth]{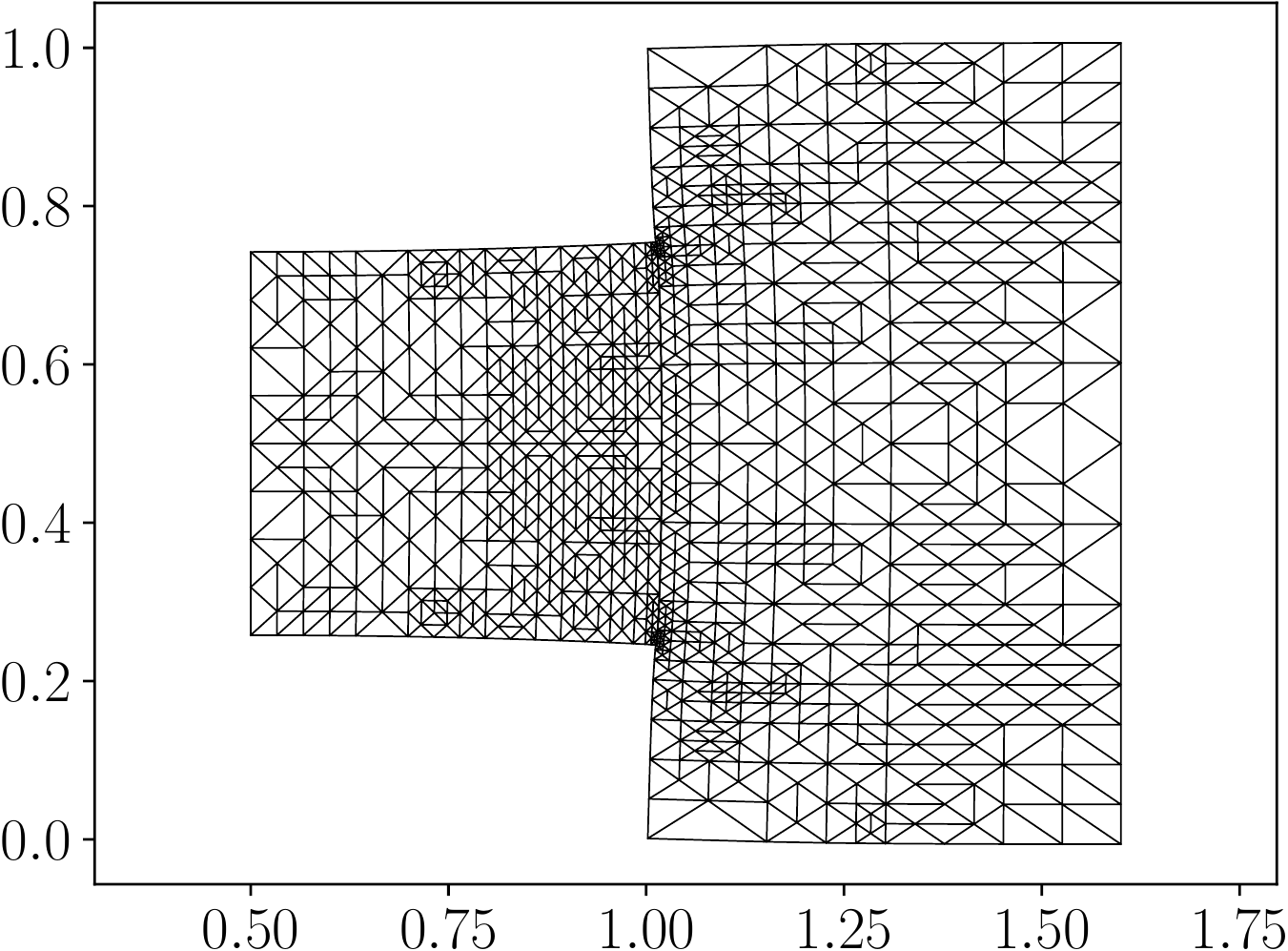}}
        \hspace{0.1cm}
        \raisebox{-0.48\height}{\includegraphics[width=0.50\textwidth]{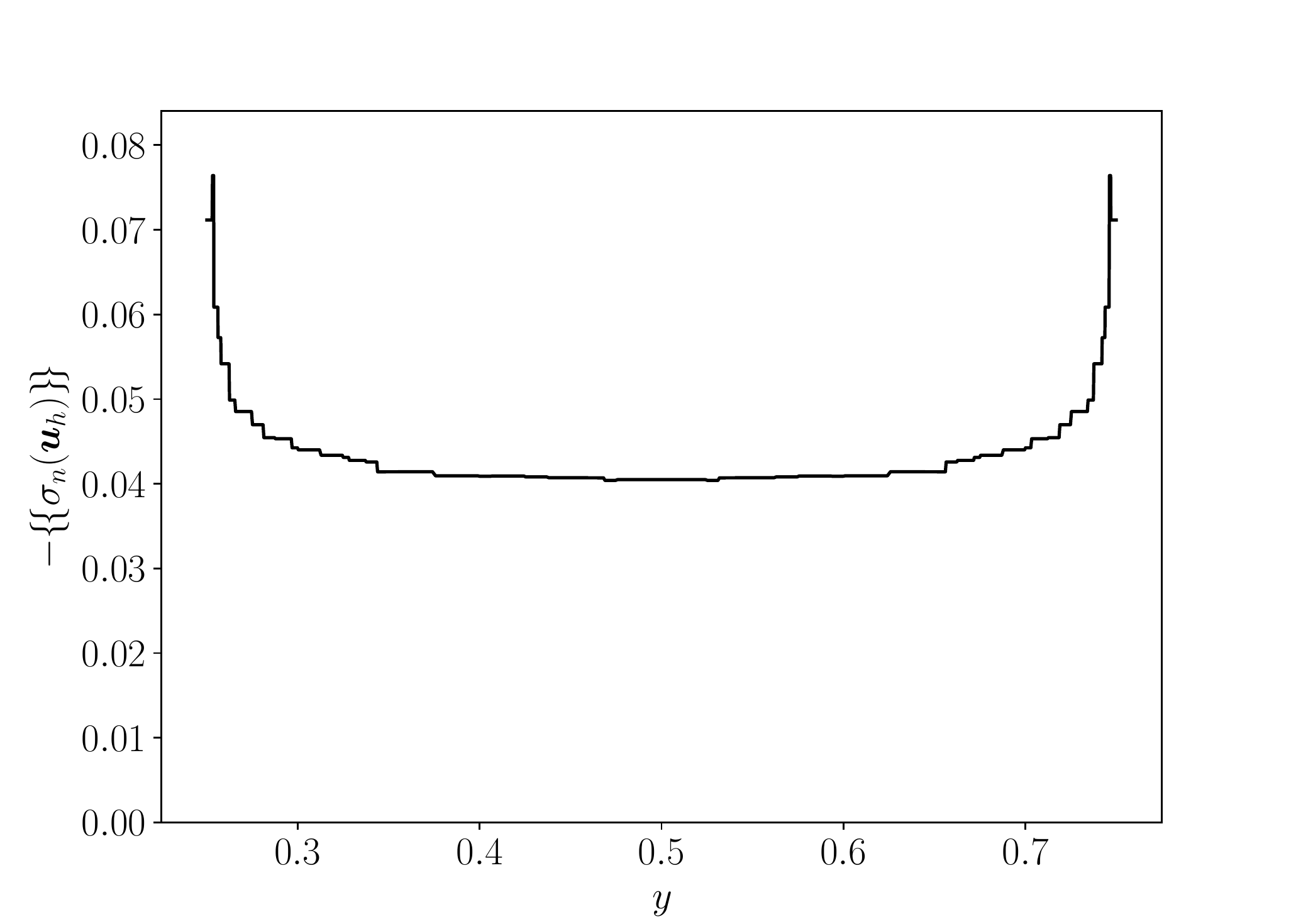}}
        \caption{$P_1$ after 8 adaptive refinements.}
        \label{fig:adaptseq}
    \end{subfigure}
    \begin{subfigure}[t]{\textwidth}
      \centering
      \raisebox{-0.5\height}{\includegraphics[width=0.4\textwidth]{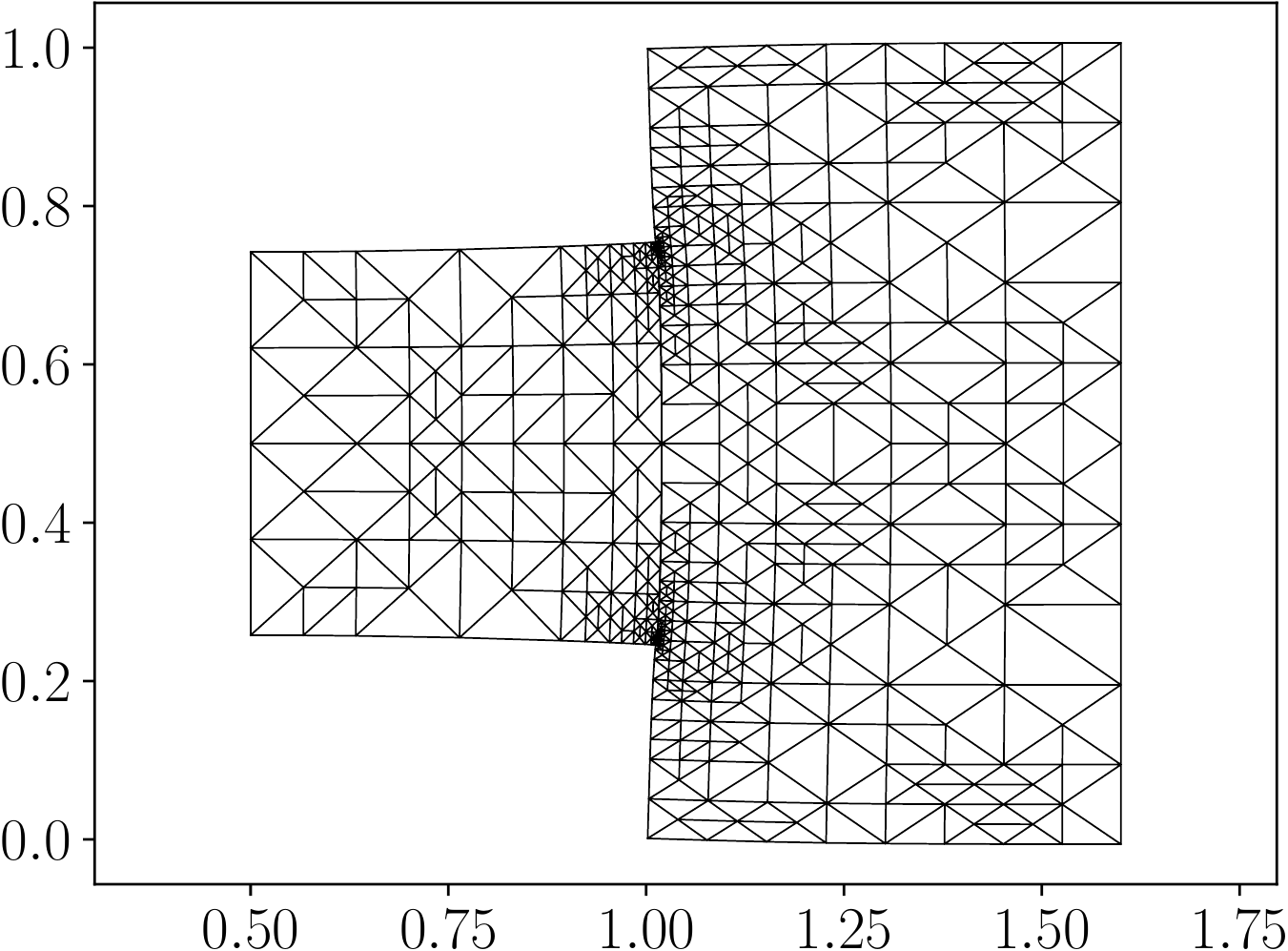}}
      \hspace{0.1cm}
      \raisebox{-0.48\height}{\includegraphics[width=0.5\textwidth]{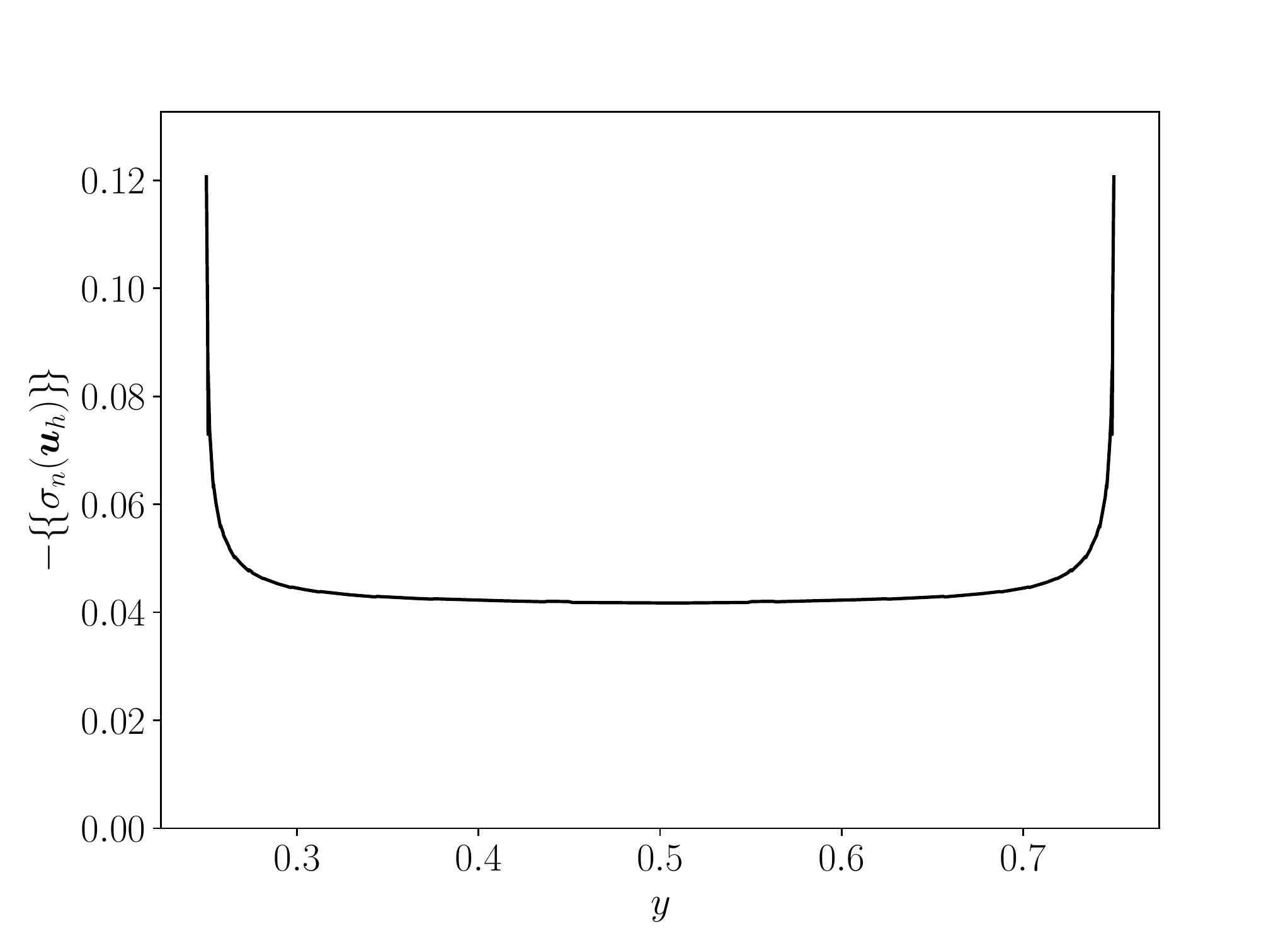}}
      \caption{$P_2$ after 8 adaptive refinements.}
      \label{fig:adaptseq1}
    \end{subfigure}
    \begin{subfigure}[t]{\textwidth}
      \vspace{0.5cm}
      \centering
      \begin{tikzpicture}[scale=0.9]
        \begin{axis}[
                xmode = log,
                ymode = log,
                xlabel = {$N$},
                ylabel = {$\eta+S$},
                grid = both,
                legend style={at={( 0.4,0.83)}, anchor=south west},
            ]
            \addplot table[only marks,x=ndofs,y=eta] {\adaptPII};
            \addplot table[only marks,x=ndofs,y=eta] {\adaptPI};
            \addplot table[only marks,x=ndofs,y=eta] {\unifPII};
            \addplot table[only marks,x=ndofs,y=eta] {\unifPI};

            \addplot[blue] table[y={create col/linear regression={y=eta}}] {\adaptPII};
            \xdef\adaptPIIcoeff{\pgfplotstableregressiona};

            \addplot[red] table[y={create col/linear regression={y=eta}}] {\adaptPI};
            \xdef\adaptPIcoeff{\pgfplotstableregressiona};

            \addplot[brown] table[y={create col/linear regression={y=eta}}] {\unifPII};
            \xdef\unifPIIcoeff{\pgfplotstableregressiona};

            \addplot[black] table[y={create col/linear regression={y=eta}}] {\unifPI};
            \xdef\unifPIcoeff{\pgfplotstableregressiona};

            \addlegendentry{Adaptive $P_2$, $\mathcal{O}(N^{\pgfmathprintnumber{\adaptPIIcoeff}})$}
            \addlegendentry{Adaptive $P_1$, $\mathcal{O}(N^{\pgfmathprintnumber{\adaptPIcoeff}})$}
            \addlegendentry{Uniform $P_2$, $\mathcal{O}(N^{\pgfmathprintnumber{\unifPIIcoeff}})$}
            \addlegendentry{Uniform $P_1$, $\mathcal{O}(N^{\pgfmathprintnumber{\unifPIcoeff}})$}
        \end{axis}
      \end{tikzpicture}
      \caption{The convergence rates of the total error estimator $\eta+S$ as a function of the number of degrees-of-freedom $N$.}
      \label{fig:adaptconv}
    \end{subfigure}
    \caption{Block against a block example.}
\end{figure}

\pgfplotstableread{
  ndofs eta
  272   0.011183207 
  398  5.7337908e-3 
  486  5.3547940e-3 
  684  4.1724822e-3 
  860  3.5950081e-3 
 1106  2.7538947e-3 
 1534  1.9910455e-3 
 2242  1.4039771e-3 
 2774  1.1818160e-3 
}\adaptPIIvII

\pgfplotstableread{
  ndofs eta
      82   0.016158947 
   132   0.017240179 
   180   0.016364739 
   288   0.012653927 
   454   0.011149878 
   626  9.0995024e-3 
   934  7.8357068e-3 
  1252  6.8366626e-3 
  1730  5.7220072e-3 
  2596  4.5997457e-3 
  3412  4.0745364e-3 
  4704  3.3075068e-3 
  6314  2.8390928e-3 
  7960  2.5532286e-3 
 10534  2.2768244e-3 
 14296  1.9378898e-3 
}\adaptPIvII

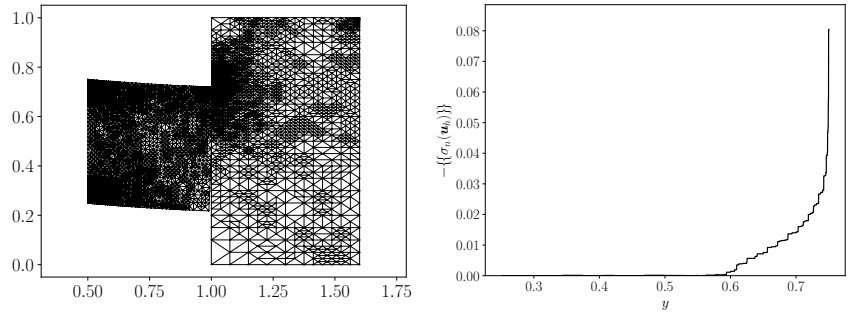
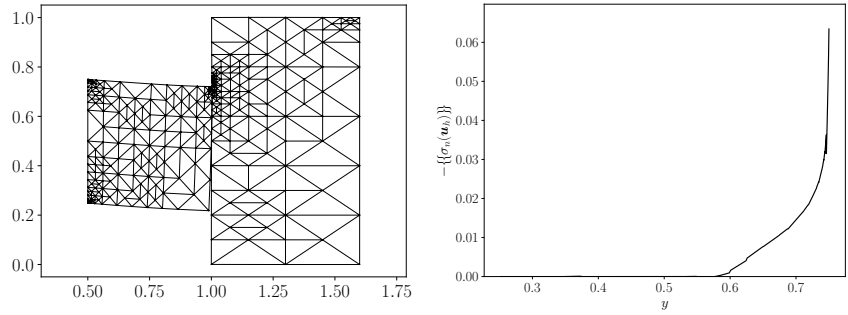
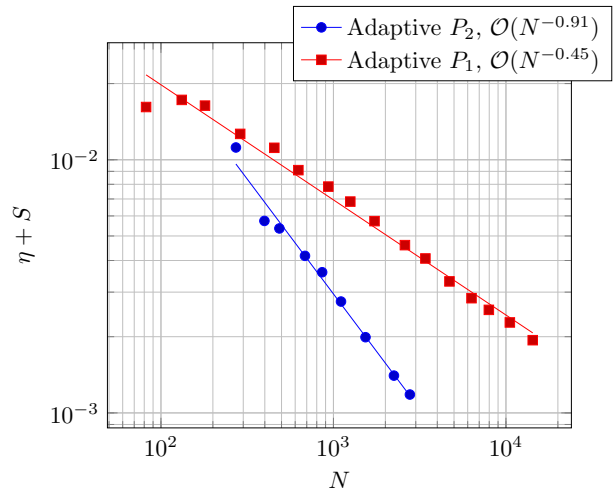
\begin{figure}[h!]
    \begin{subfigure}[t]{\textwidth}
        \centering
        \raisebox{-0.5\height}{\includegraphics[width=0.42\textwidth]{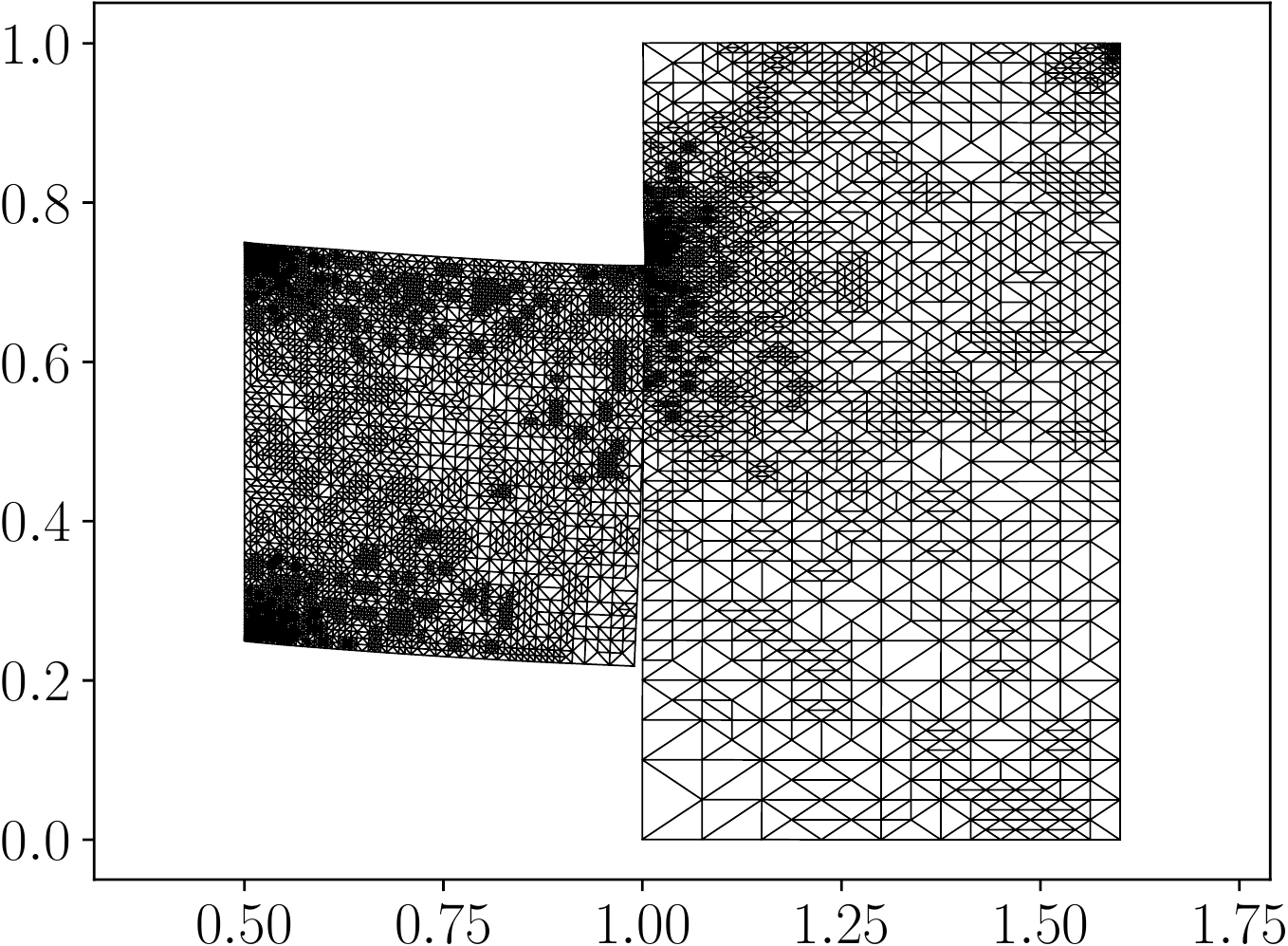}}
        \hspace{0.1cm}
        \raisebox{-0.52\height}{\includegraphics[width=0.43\textwidth]{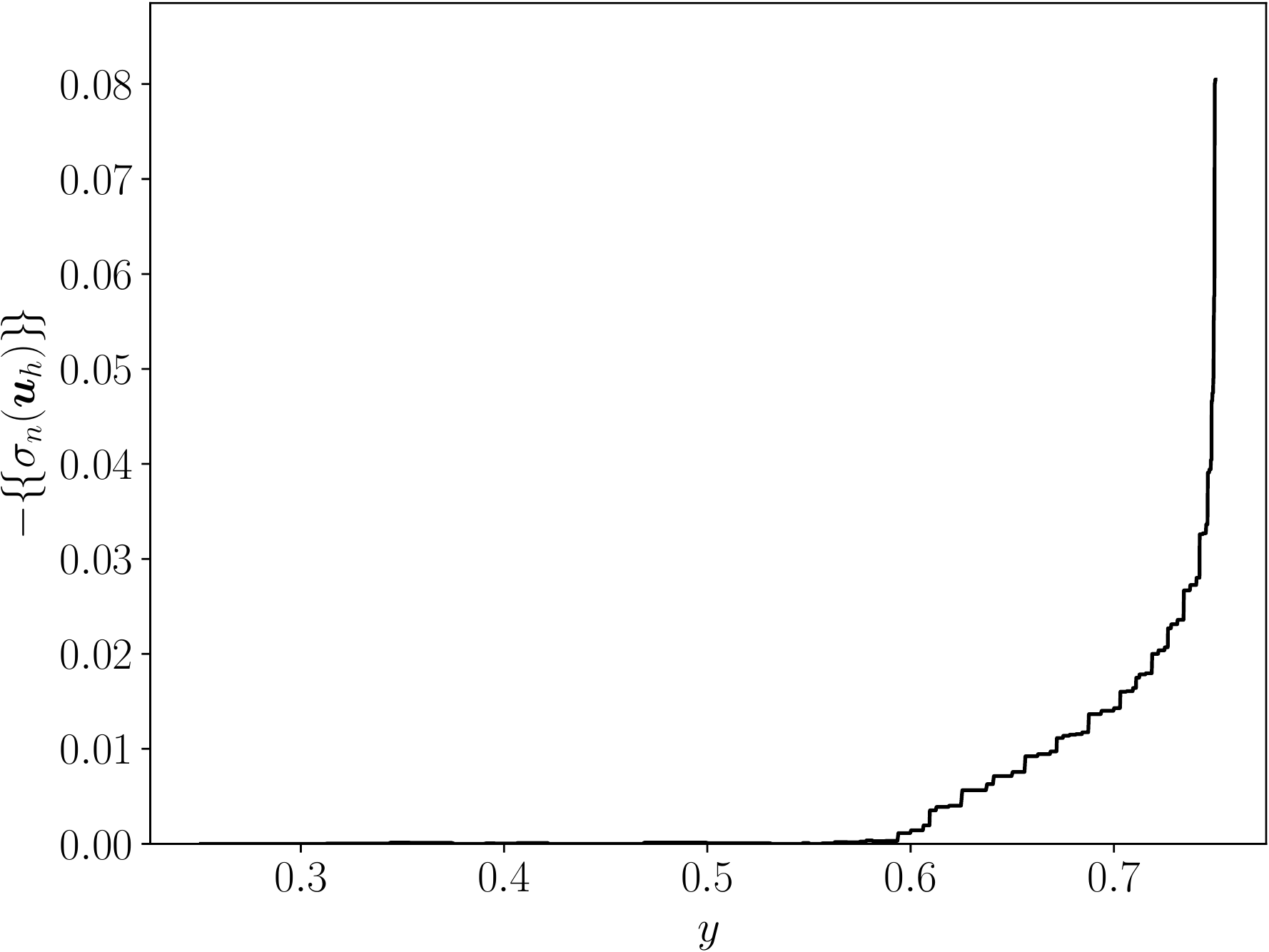}}
        \caption{$P_1$ after 15 adaptive refinements.}
            \label{fig:adaptseq2}
    \end{subfigure}
    \begin{subfigure}[t]{\textwidth}
      \centering
      \vspace{0.2cm}
      \raisebox{-0.5\height}{\includegraphics[width=0.42\textwidth]{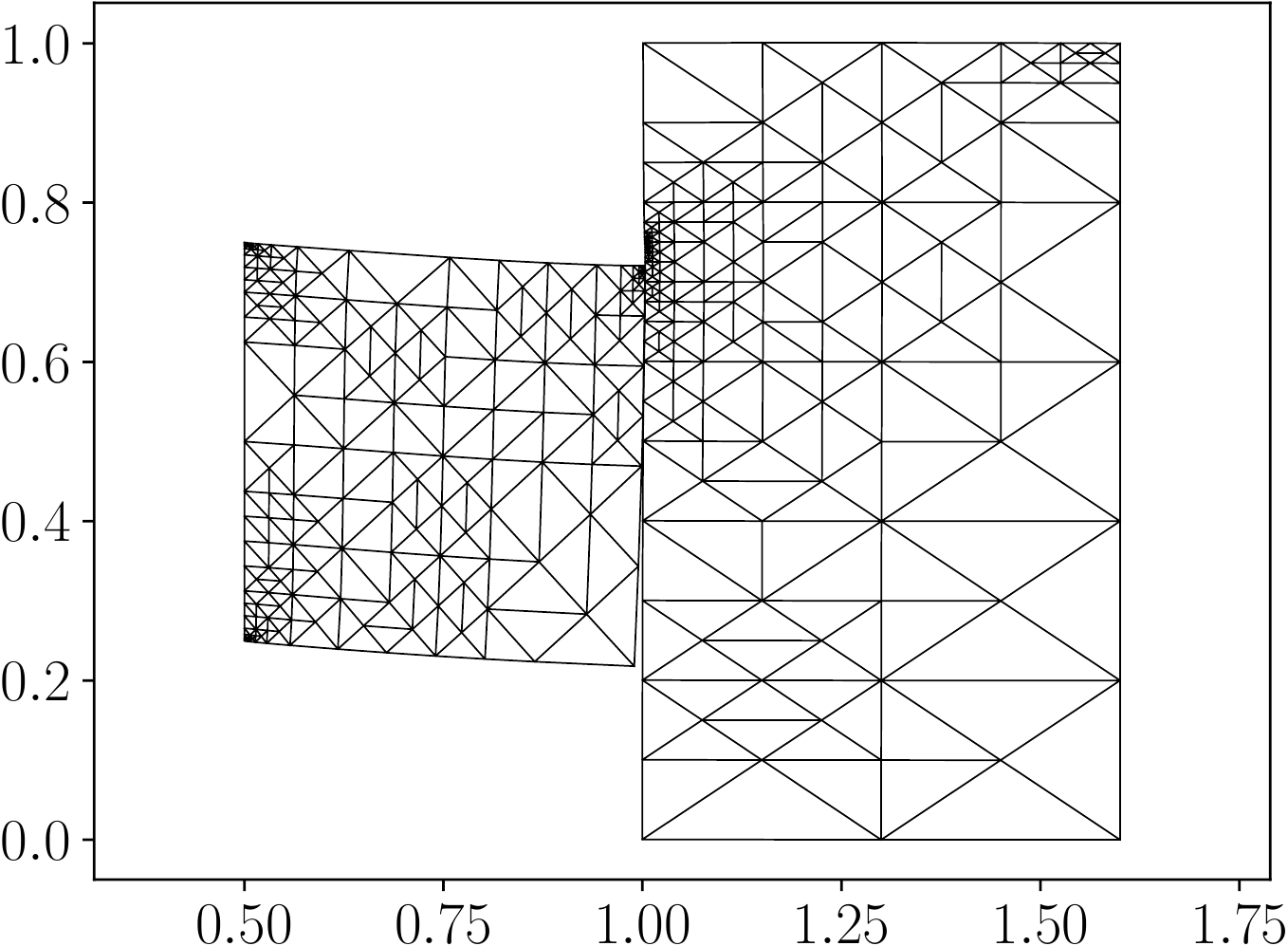}}
      \hspace{0.1cm}
      \raisebox{-0.52\height}{\includegraphics[width=0.43\textwidth]{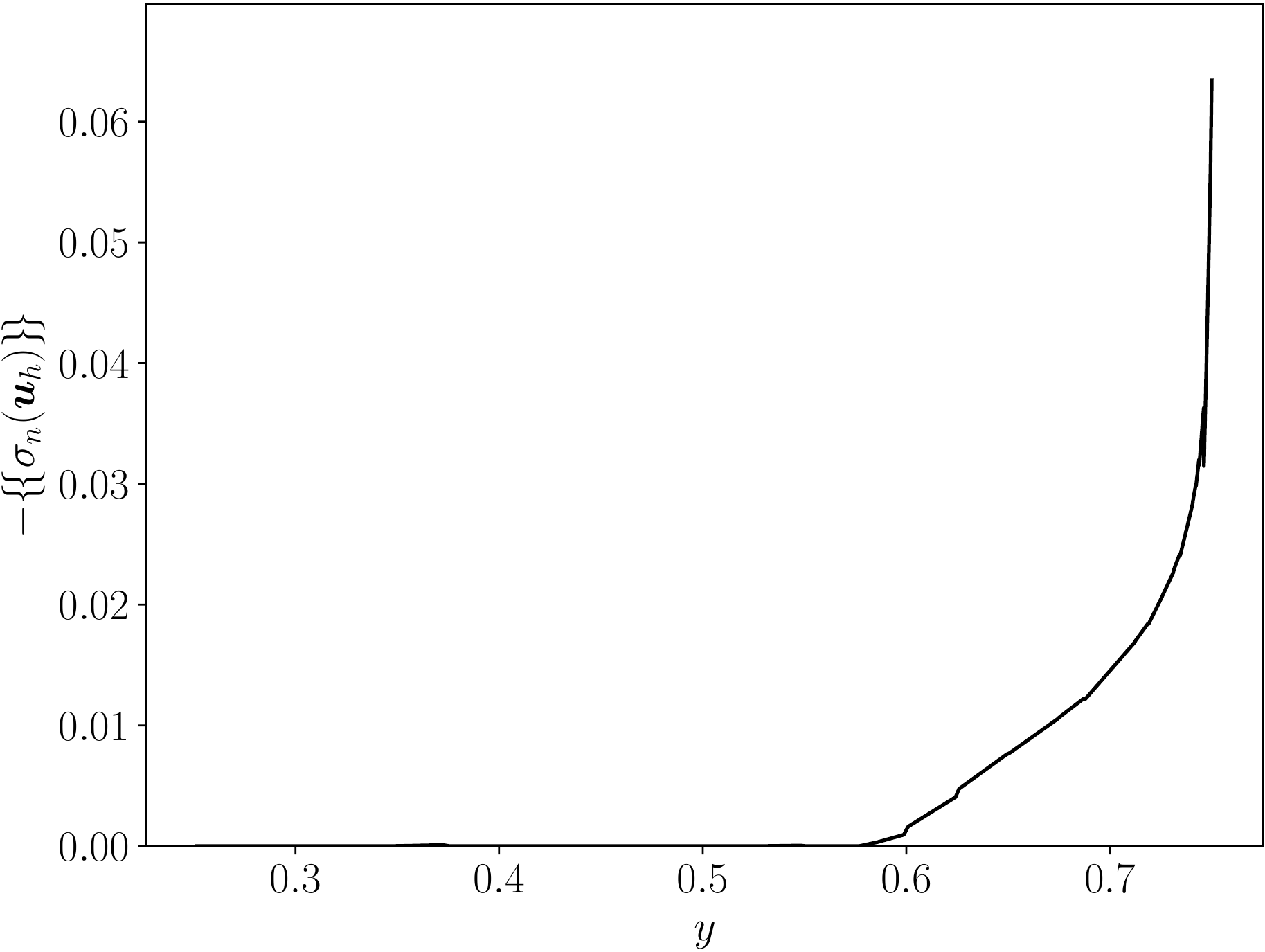}}
      \caption{$P_2$ after 8 adaptive refinements.}
      \label{fig:adaptseq21}
    \end{subfigure}
    \begin{subfigure}[t]{\textwidth}
      \vspace{0.5cm}
      \centering
      \begin{tikzpicture}[scale=0.9]
        \begin{axis}[
            xmode = log,
            ymode = log,
            xlabel = {$N$},
            ylabel = {$\eta+S$},
            grid = both,
            legend style={at={( 0.4,0.9)}, anchor=south west},
          ]
          \addplot table[only marks,x=ndofs,y=eta] {\adaptPIIvII};
          \addplot table[only marks,x=ndofs,y=eta] {\adaptPIvII};

          \addplot[blue] table[y={create col/linear regression={y=eta}}] {\adaptPIIvII};
          \xdef\adaptPIIcoeff{\pgfplotstableregressiona};

          \addplot[red] table[y={create col/linear regression={y=eta}}] {\adaptPIvII};
          \xdef\adaptPIcoeff{\pgfplotstableregressiona};

          \addlegendentry{Adaptive $P_2$, $\mathcal{O}(N^{\pgfmathprintnumber{\adaptPIIcoeff}})$}
          \addlegendentry{Adaptive $P_1$, $\mathcal{O}(N^{\pgfmathprintnumber{\adaptPIcoeff}})$}
        \end{axis}
      \end{tikzpicture}
      \caption{The convergence rates of the total error estimator $\eta+S$ as a function of the number of degrees-of-freedom $N$.}
          \label{fig:adaptconv2}
    \end{subfigure}
    \caption{Downward bending block example.}
\end{figure}

\pgfplotstableread{
  ndofs eta
272 0.012723237270218391
434 0.006000909076803074
558 0.004583480175542396
828 0.003936495876318481
1164 0.003160109055702916
1664 0.002010452328317928
2068 0.0015608484414634412
2556 0.0013024813566649249
3206 0.0010417489800942032
4104 0.0007613557854558165
5388 0.0006110541916976908
}\adaptPIIvIII

\pgfplotstableread{
  ndofs eta
272 0.012228070065093745
382 0.007225426461337643
506 0.0055660124532230255
670 0.0038723402971396234
872 0.002955825358875367
1084 0.00245695561213524
1504 0.0014878670784207713
1824 0.00126483538779943
2376 0.0009397145901848334
3014 0.0007748331056696119
3642 0.0006573441979683001
}\adaptPIIvIIII

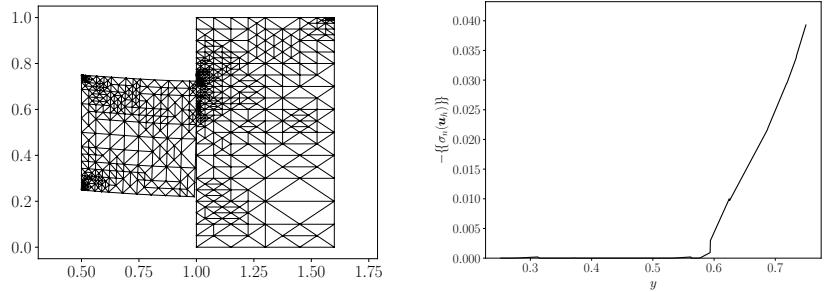
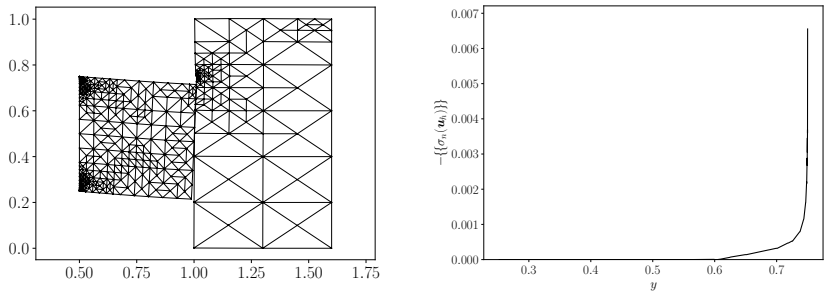
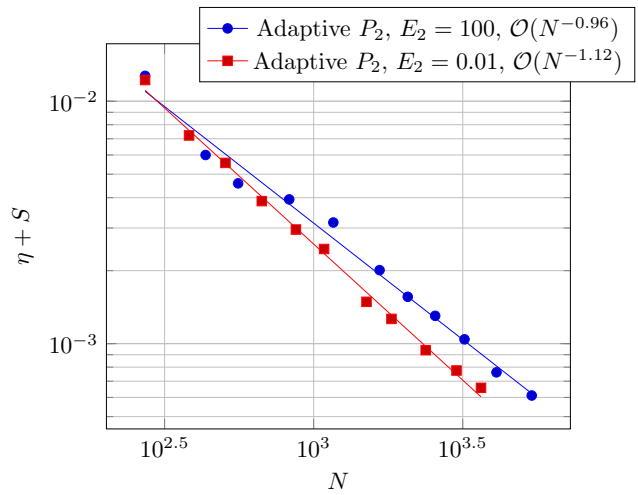
\begin{figure}[h!]
    \begin{subfigure}[t]{\textwidth}
      \centering
      \includegraphics[width=0.46\textwidth]{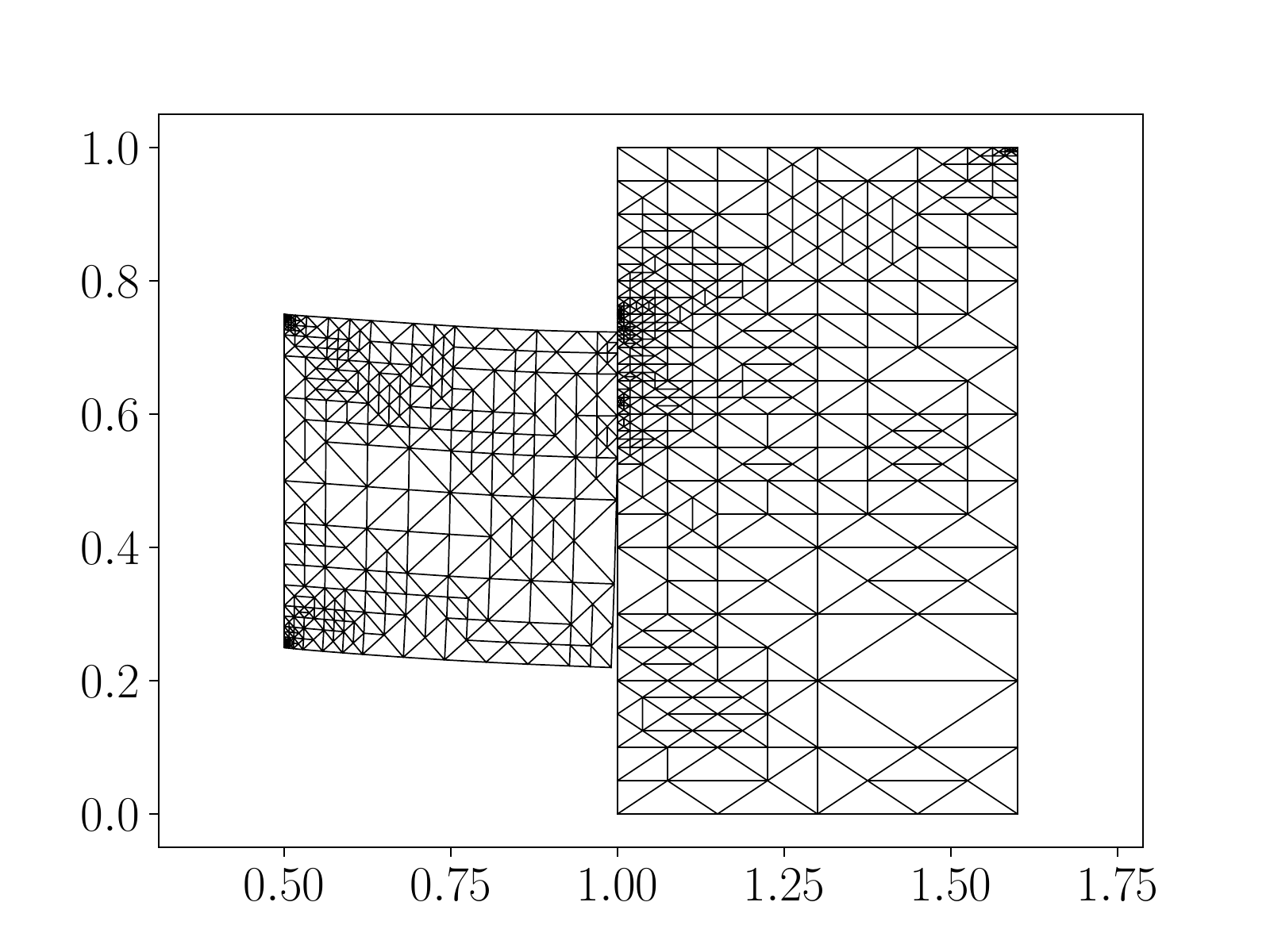}
      \hspace{-0.1cm}
      \includegraphics[width=0.455\textwidth]{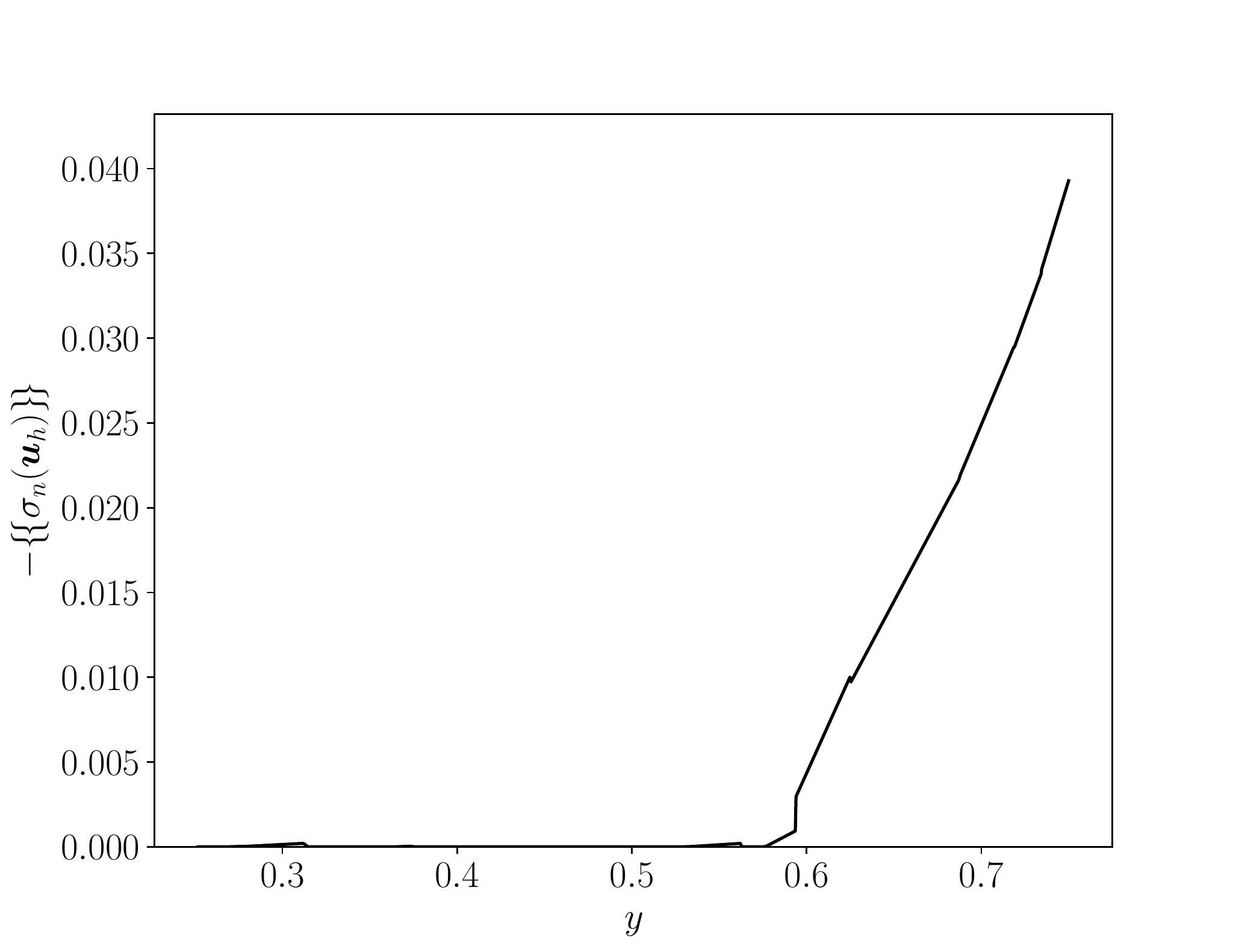}
      \caption{$P_2$ after 10 adaptive refinements with $E_2=100$.}
      \label{fig:adaptseq2E100}
    \end{subfigure}
    \begin{subfigure}[t]{\textwidth}
        \centering
        \includegraphics[width=0.46\textwidth]{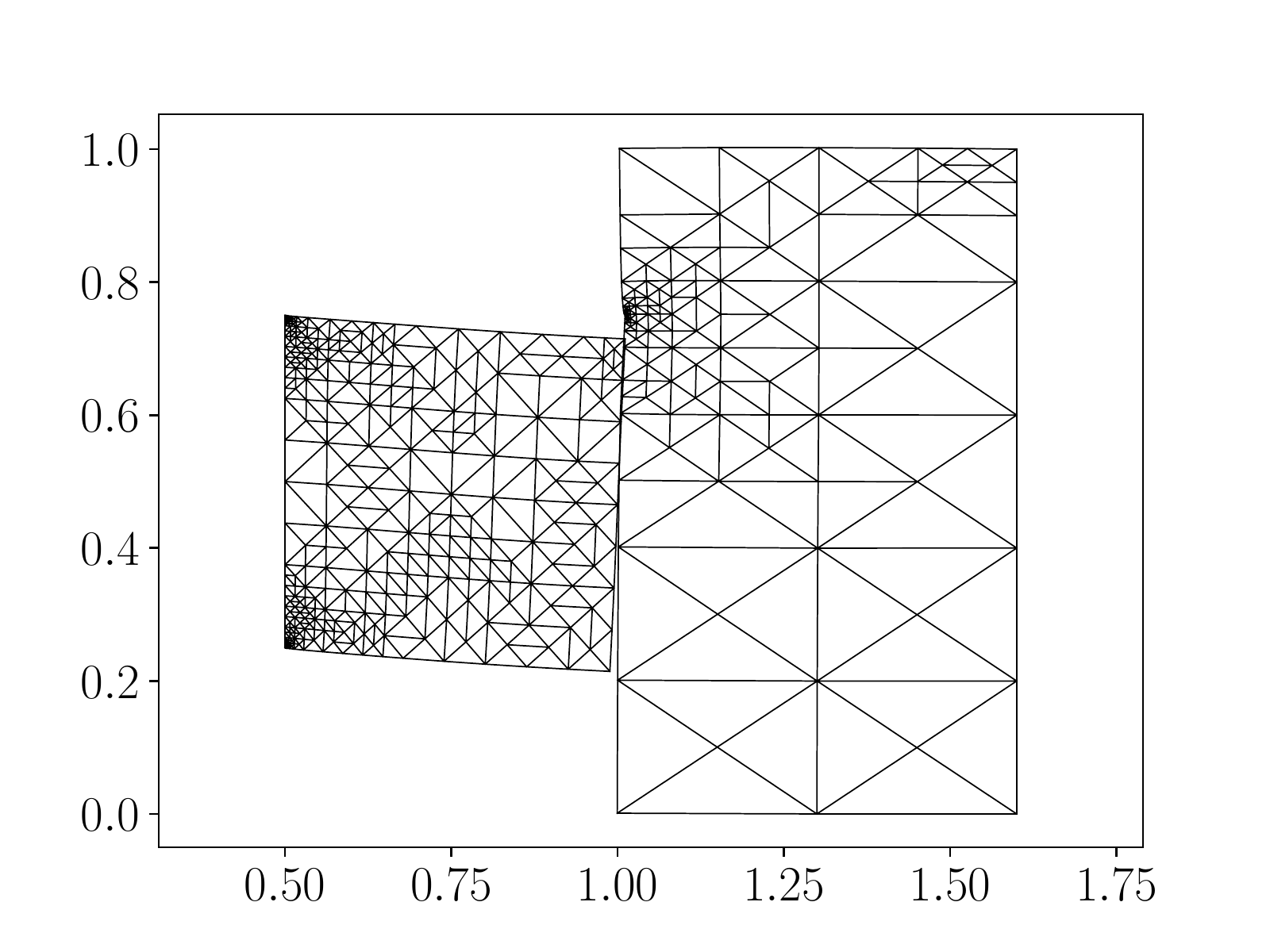}
        \hspace{-0.1cm}
        \includegraphics[width=0.46\textwidth]{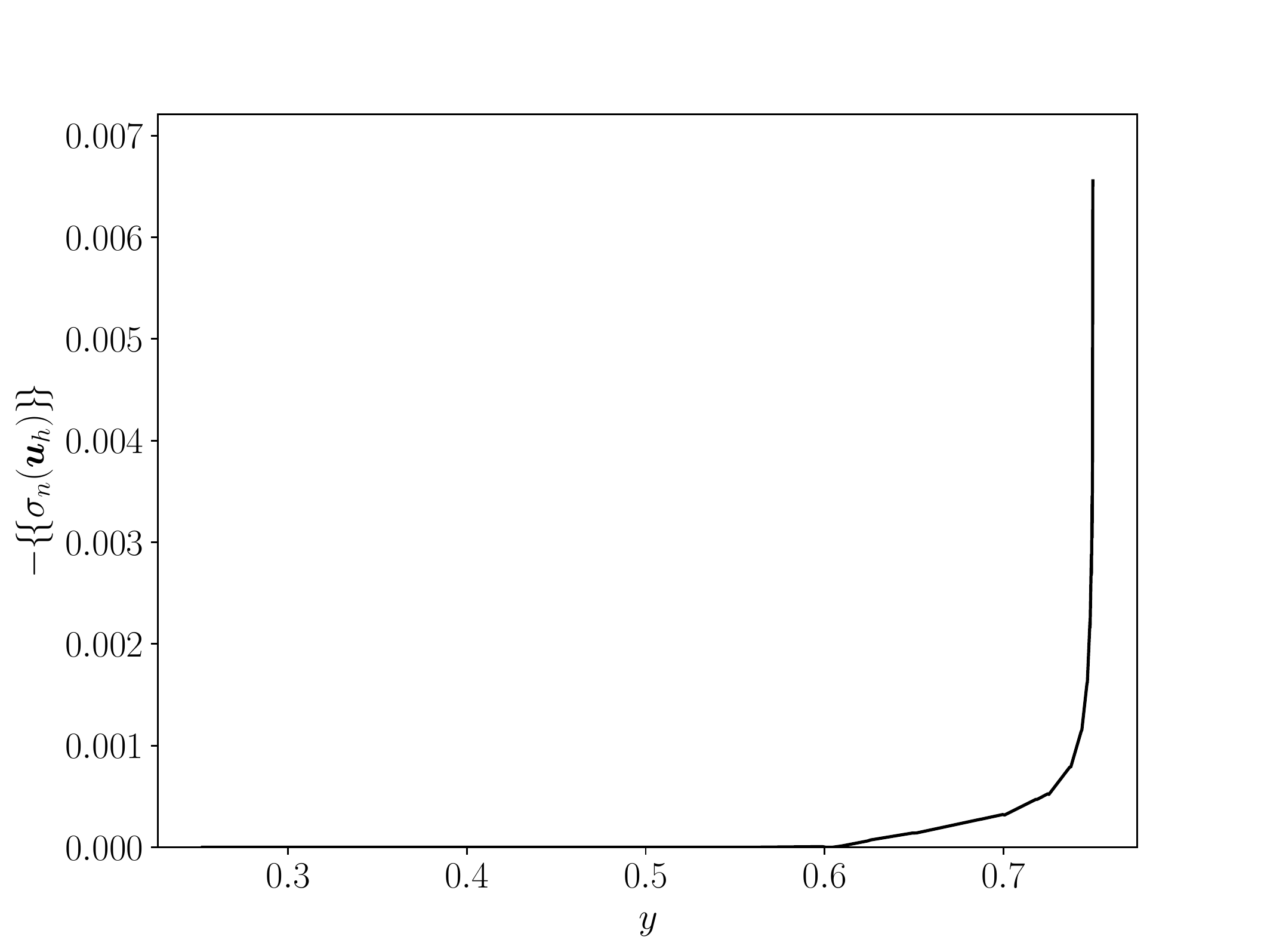}
        \caption{$P_2$ after 10 adaptive refinements with $E_2=0.01$.}
            \label{fig:adaptseq2E001}
    \end{subfigure}
    \begin{subfigure}[t]{\textwidth}
      \vspace{0.5cm}
      \centering
      \begin{tikzpicture}[scale=0.9]
        \begin{axis}[
            xmode = log,
            ymode = log,
            xlabel = {$N$},
            ylabel = {$\eta+S$},
            grid = both,
            legend style={at={( 0.2,0.9)}, anchor=south west},
          ]
          \addplot table[only marks,x=ndofs,y=eta] {\adaptPIIvIII};
          \addplot table[only marks,x=ndofs,y=eta] {\adaptPIIvIIII};

          \addplot[blue] table[y={create col/linear regression={y=eta}}] {\adaptPIIvIII};
          \xdef\adaptPIIcoeff{\pgfplotstableregressiona};

          \addplot[red] table[y={create col/linear regression={y=eta}}] {\adaptPIIvIIII};
          \xdef\adaptPIcoeff{\pgfplotstableregressiona};

          \addlegendentry{Adaptive $P_2$, $E_2=100$, $\mathcal{O}(N^{\pgfmathprintnumber{\adaptPIIcoeff}})$}
          \addlegendentry{Adaptive $P_2$, $E_2=0.01$, $\mathcal{O}(N^{\pgfmathprintnumber{\adaptPIcoeff}})$}
        \end{axis}
      \end{tikzpicture}
      \caption{The convergence rates of the total error estimator $\eta+S$ as a function of the number of degrees-of-freedom $N$.}
    \end{subfigure}
    \caption{Effect of a jump in the Young's modulus.}
    \label{fig:adaptconv2E}
\end{figure}

\pgfplotstableread{
  ndofs eta
272 0.012778081538569214
398 0.007015126598645964
558 0.004990357010696439
726 0.004230414250189424
940 0.003481634919647768
1444 0.002281647016953714
1894 0.00181000991207087
2508 0.0014709920471925792
3280 0.001165013679949896
4352 0.0008942708179408232
5766 0.0006686948407056319
}\adaptPIIalphaZONE

\pgfplotstableread{
  ndofs eta
272 0.013283856137114143
398 0.008811639759331278
486 0.0065464847046348374
684 0.00502158259569199
860 0.0041124725508012204
1106 0.0031563601691198926
1534 0.0022615538945431855
2242 0.0015785278124859684
2774 0.0013015962689135452
3474 0.0010155305472788823
4548 0.0007958543988012092
}\adaptPIIalphaZZONE

\pgfplotstableread{
  ndofs eta
272 0.013454436004200796
434 0.007791152977011452
594 0.005640766224680968
746 0.004398459873787459
1008 0.003613153124928436
1074 0.003322563685093166
1444 0.002578970397746524
1502 0.0024045567023091136
1952 0.0018571870932495823
2040 0.0017719330837594456
2454 0.0015114029856898529
}\adaptPIIalphaZZZONE

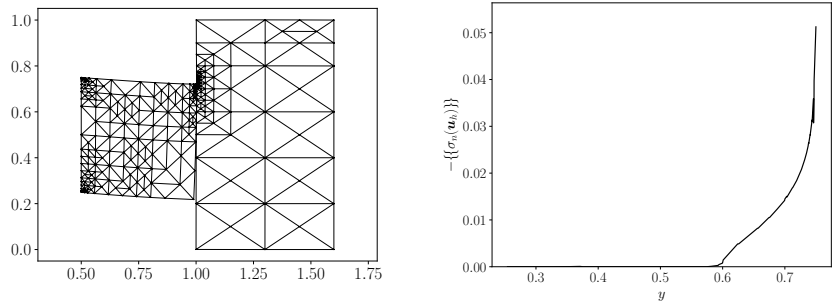
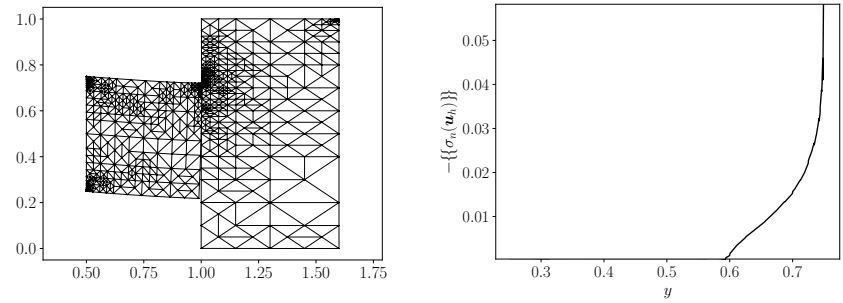
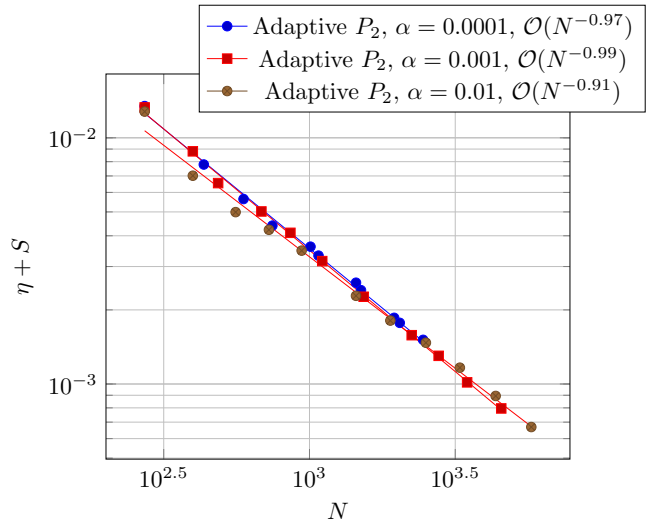
\begin{figure}[h!]
    \begin{subfigure}[t]{\textwidth}
      \centering
      \hspace{0.1cm}
      \raisebox{-0.5\height}{\includegraphics[width=0.46\textwidth]{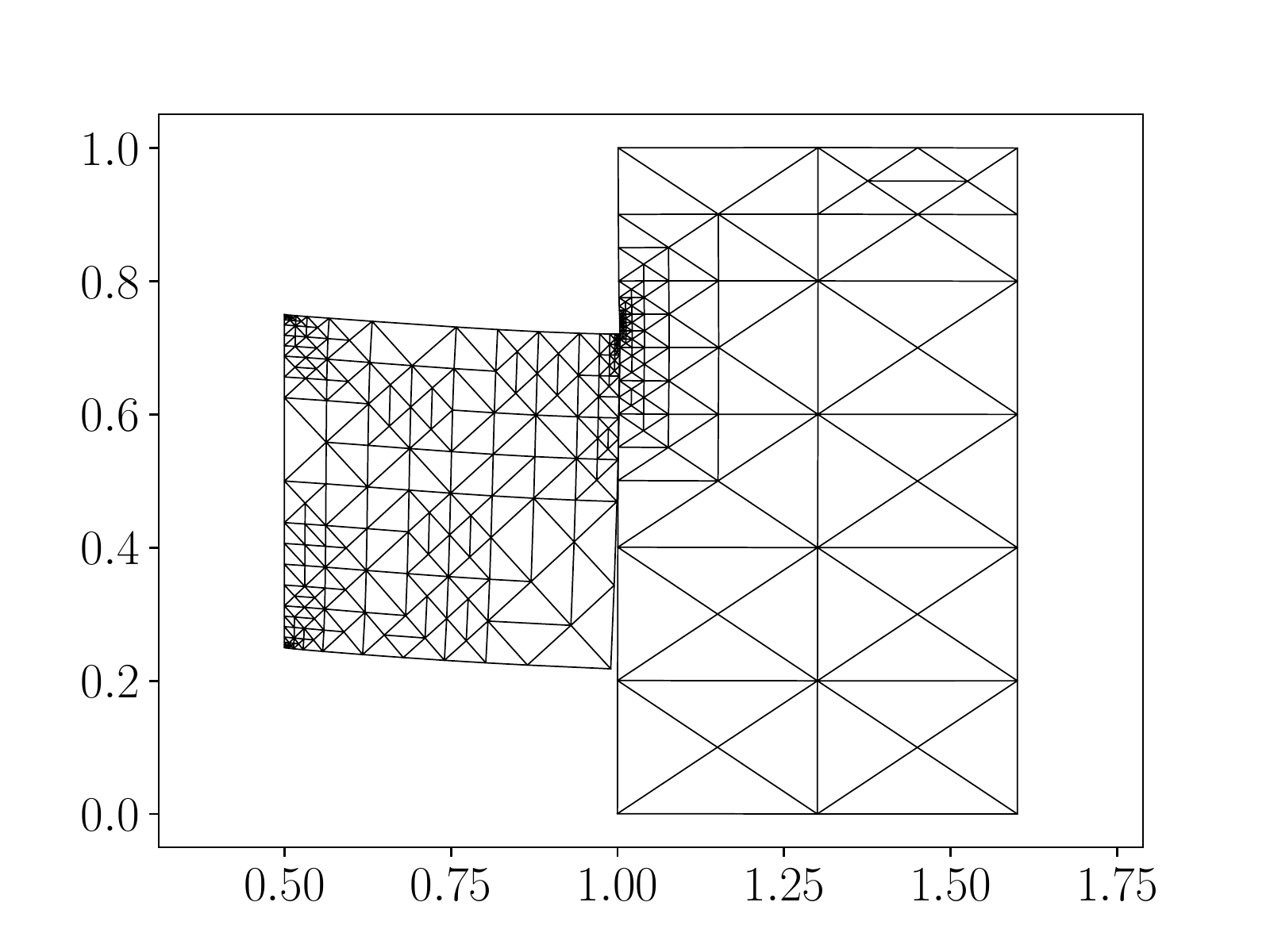}}\hspace{0.1cm}
      \raisebox{-0.5\height}{\includegraphics[width=0.46\textwidth]{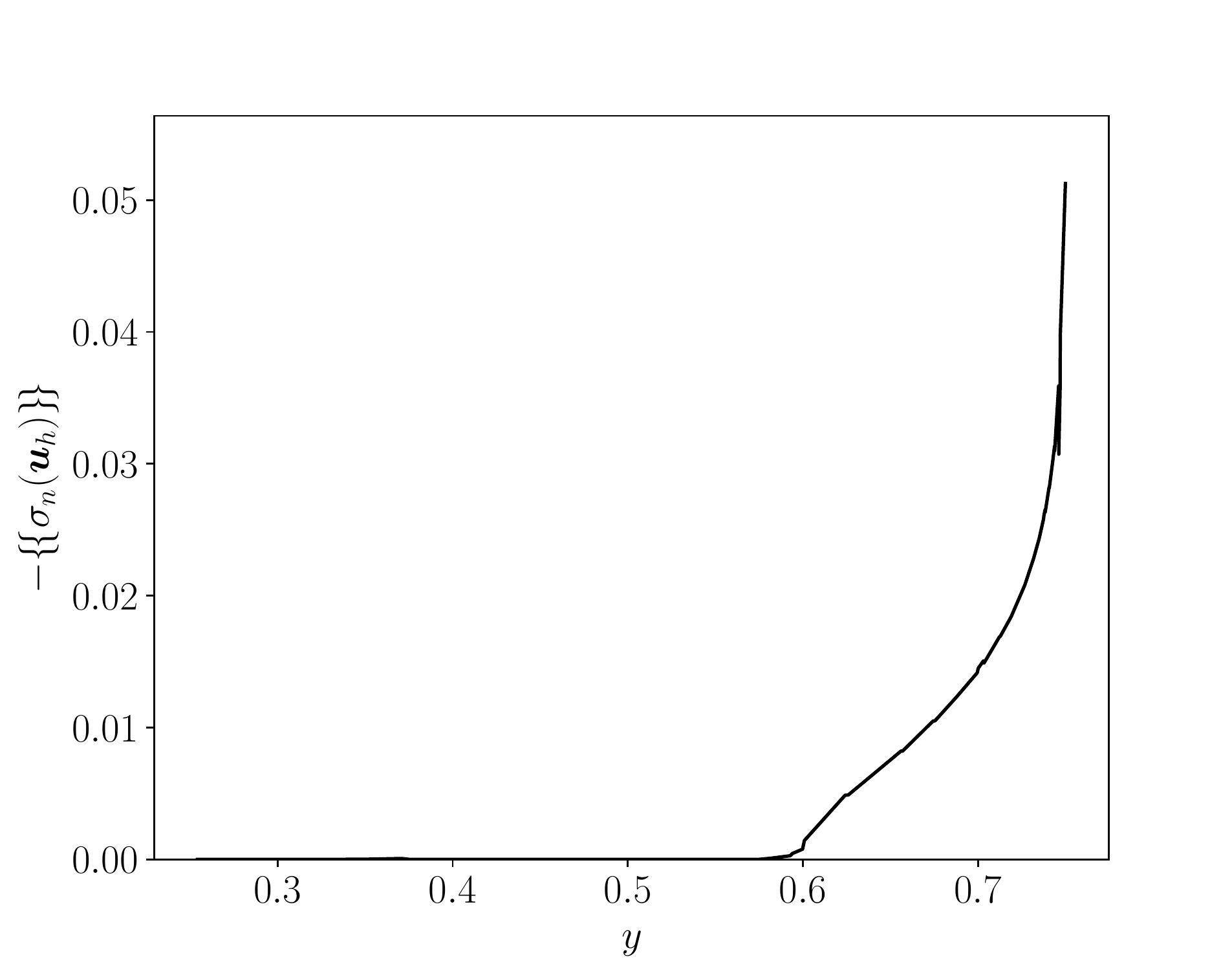}}
      \caption{$P_2$ after 10 adaptive refinements with $\alpha=0.0001$.}
      \label{fig:adaptseq2E100}
    \end{subfigure}
    \begin{subfigure}[t]{\textwidth}
      \centering
      \raisebox{-0.5\height}{\includegraphics[width=0.46\textwidth]{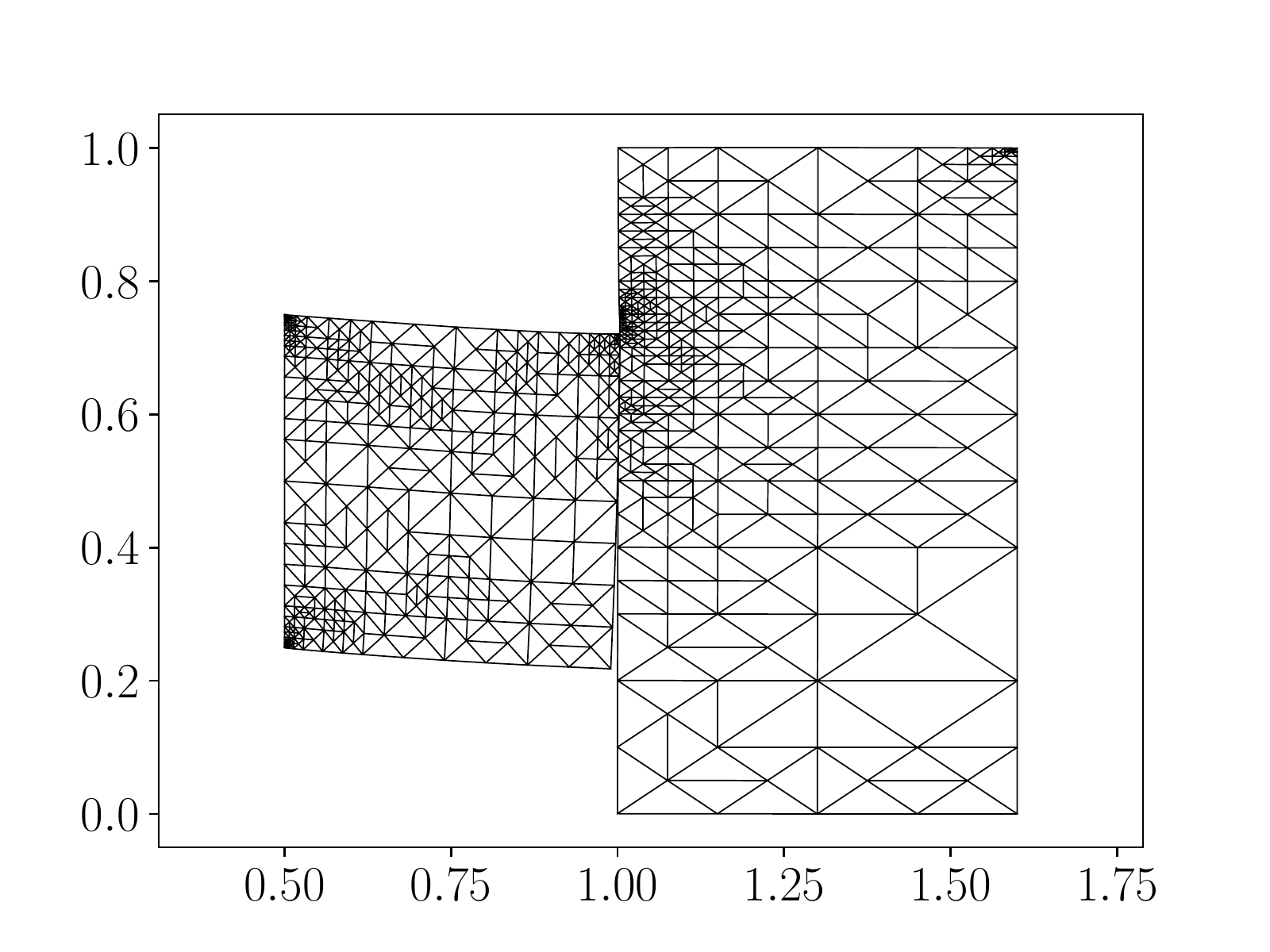}}
      \raisebox{-0.55\height}{\includegraphics[width=0.44\textwidth]{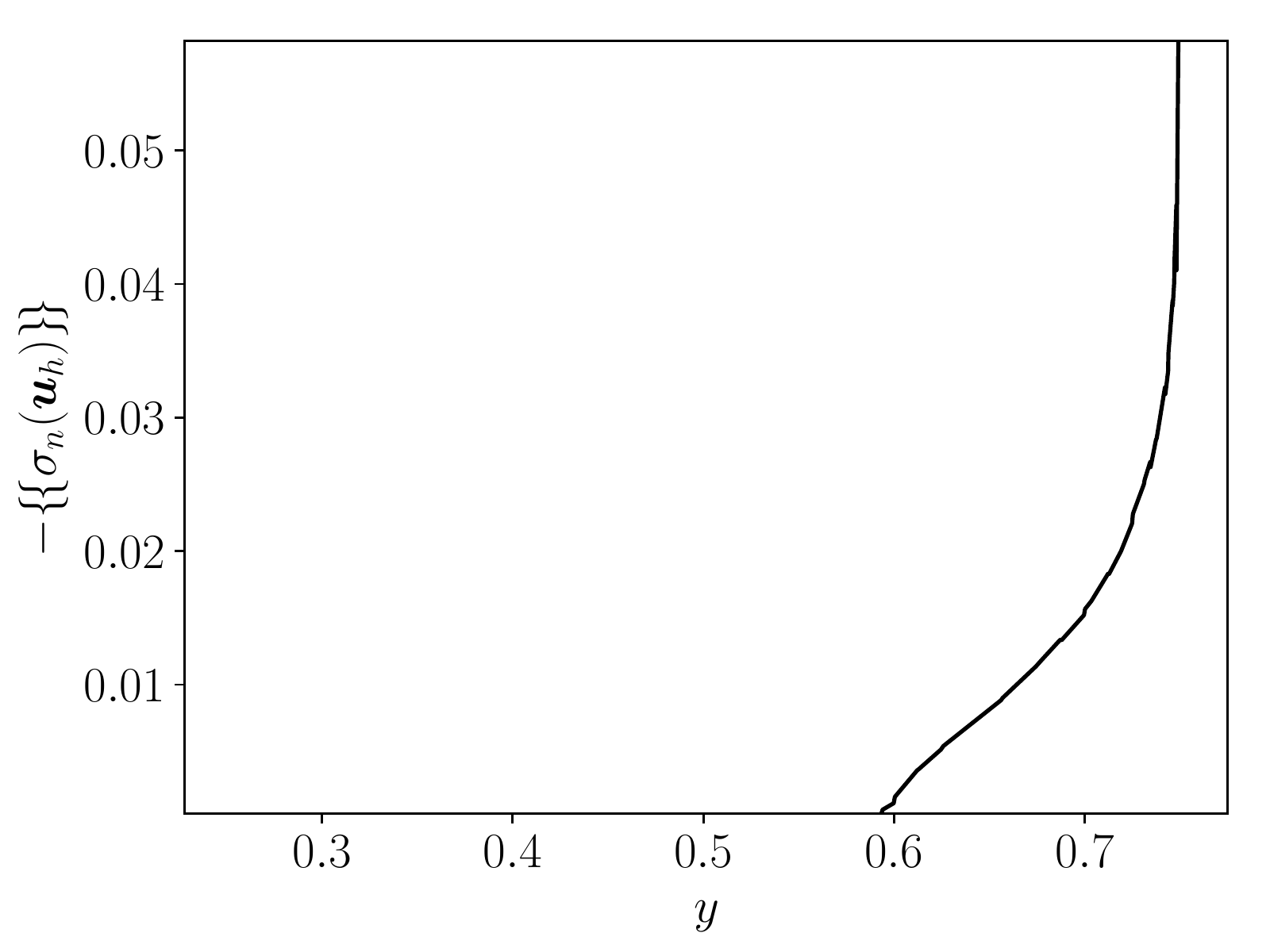}}
      \caption{$P_2$ after 10 adaptive refinements with $\alpha=0.01$.}
      \label{fig:adaptseq2E001}
    \end{subfigure}
    \begin{subfigure}[t]{\textwidth}
      \vspace{0.5cm}
      \centering
      \begin{tikzpicture}[scale=0.9]
        \begin{axis}[
            xmode = log,
            ymode = log,
            xlabel = {$N$},
            ylabel = {$\eta+S$},
            grid = both,
            legend style={at={( 0.2,0.9)}, anchor=south west},
          ]
          \addplot table[only marks,x=ndofs,y=eta] {\adaptPIIalphaZZZONE};
          \addplot table[only marks,x=ndofs,y=eta] {\adaptPIIalphaZZONE};
          \addplot table[only marks,x=ndofs,y=eta] {\adaptPIIalphaZONE};

          \addplot[blue] table[y={create col/linear regression={y=eta}}] {\adaptPIIalphaZZZONE};
          \xdef\adaptPIIcoeff{\pgfplotstableregressiona};

          \addplot[red] table[y={create col/linear regression={y=eta}}] {\adaptPIIalphaZZONE};
          \xdef\adaptPIcoeff{\pgfplotstableregressiona};

          \addplot[red] table[y={create col/linear regression={y=eta}}] {\adaptPIIalphaZONE};
          \xdef\adaptPZcoeff{\pgfplotstableregressiona};

          \addlegendentry{Adaptive $P_2$, $\alpha=0.0001$, $\mathcal{O}(N^{\pgfmathprintnumber{\adaptPIIcoeff}})$}
          \addlegendentry{Adaptive $P_2$, $\alpha=0.001$, $\mathcal{O}(N^{\pgfmathprintnumber{\adaptPIcoeff}})$}
          \addlegendentry{Adaptive $P_2$, $\alpha=0.01$, $\mathcal{O}(N^{\pgfmathprintnumber{\adaptPZcoeff}})$}
        \end{axis}
      \end{tikzpicture}
      \caption{The convergence rates of the total error estimator $\eta+S$ as a function of the number of degrees-of-freedom $N$.}
    \end{subfigure}
    \caption{Effect of changing the stabilisation parameter.}
          \label{fig:adaptconv2alpha}
\end{figure}

\begin{figure}[h!]
  \centering
  \includegraphics[width=0.49\textwidth]{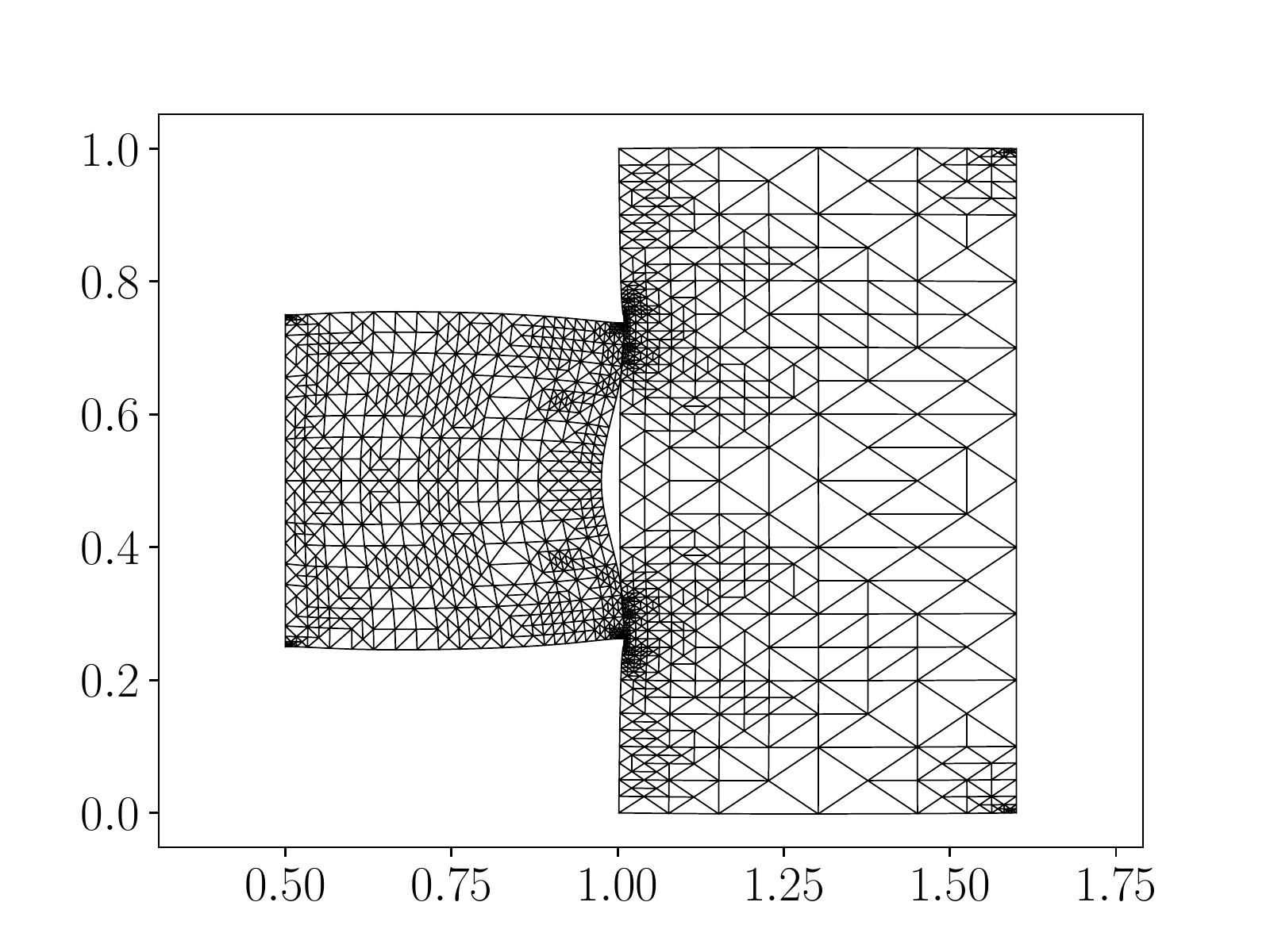}
  \includegraphics[width=0.49\textwidth]{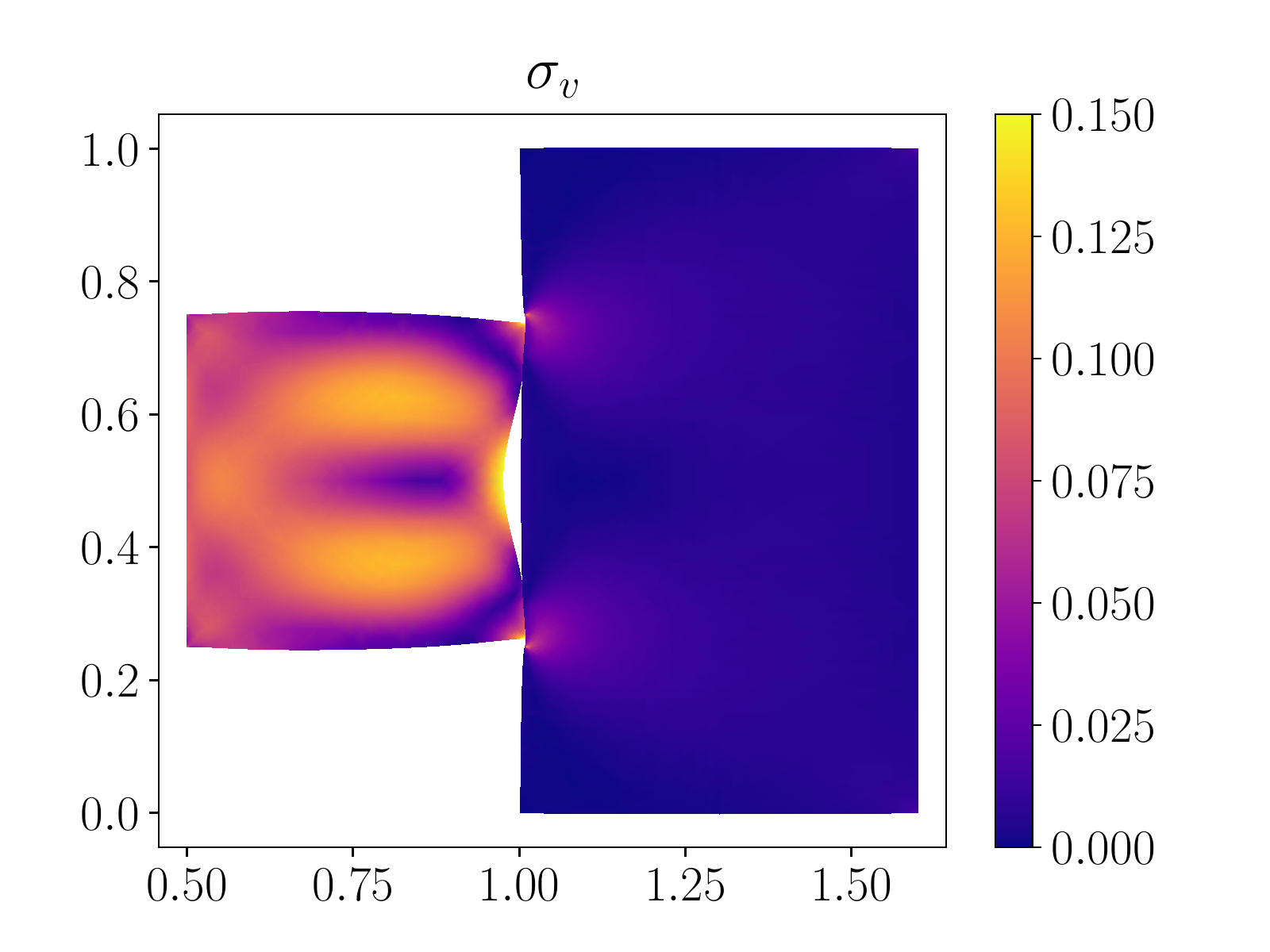}\\
        \includegraphics[width=0.49\textwidth]{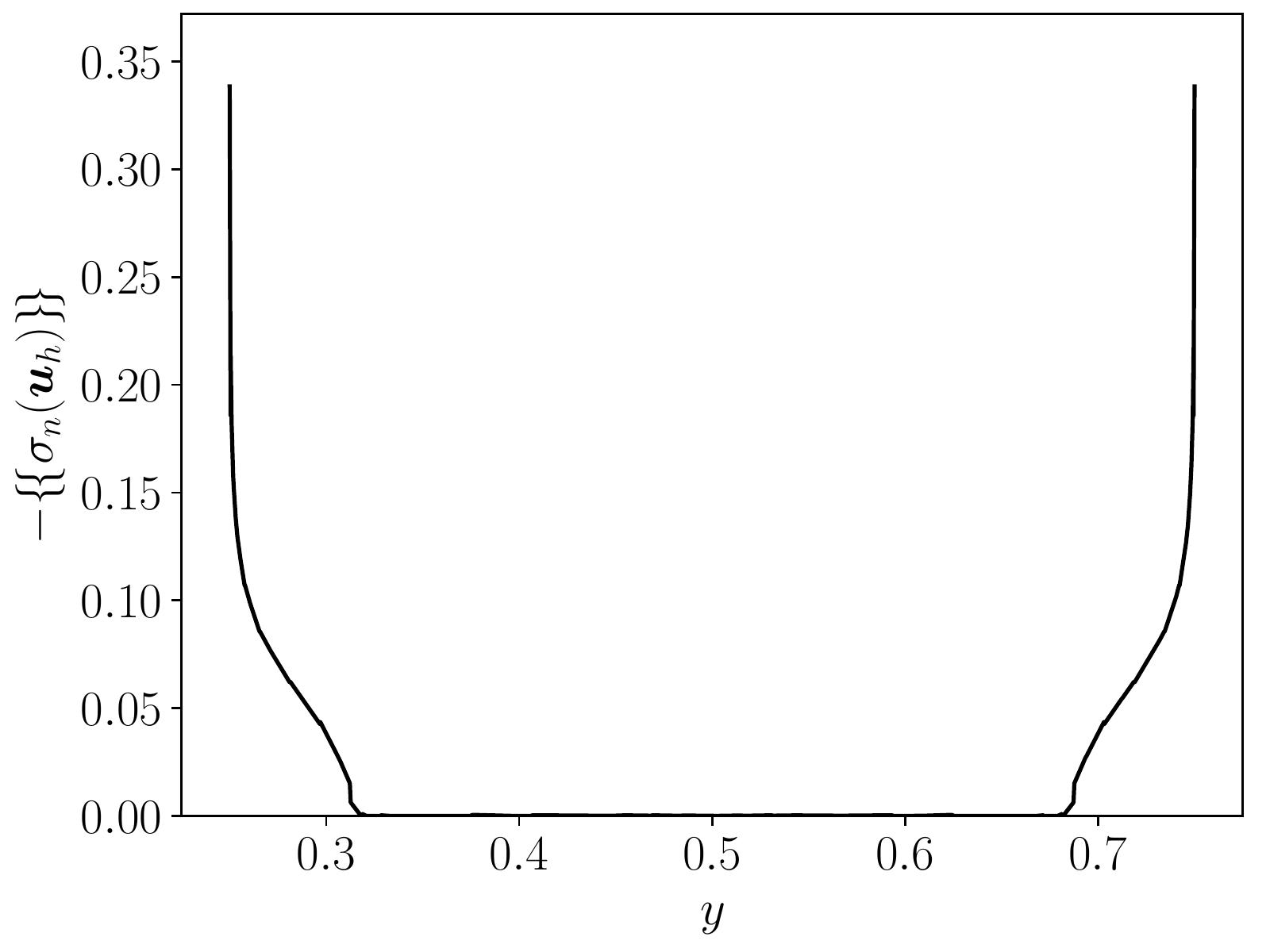}
        \caption{Example with a contact boundary consisting of two disjoint active sets, with von Mises stress $\sigma_v$ plotted in the top right figure.}
    \label{fig:final}
\end{figure}

\end{document}